\documentclass[12pt, reqno]{amsart}
\usepackage{ifthen,amsfonts,
amsmath,amssymb,
epsfig,color}

\makeatletter

\def\fnum{equation}

\begingroup

\newtheorem{theorem}{Theorem}[section]
\newtheorem{lemma}[theorem]{Lemma}
\newtheorem{propos}[theorem]{Proposition}
\newtheorem{corol}[theorem]{Corollary}
\endgroup

\newtheorem{definition}[theorem]{Definition}

\newtheorem{remark}[theorem]{Remark}

\numberwithin{equation}{section}

\newcommand{\Wedge}{{\rm Wed}\,}
\newcommand{\an}{{\rm An}}
\newcommand{\eps}{{\varepsilon}}

\newcommand{\diam}{{\text {diam}}}
\newcommand{\dist}{{\text {dist}}}

\newcommand{\R}{{\rm R}}

\def\RR{{\bf  R}}

\def\SS{{\bf  S}}

\newcommand{\Length}{{\text {Length}}}

\newcommand\diff{{\rm Diff}_0}

\newcommand\D{{\mathcal{D}}}

\newcommand\B{{\mathcal{B}}}

\def\rn#1{{\bf R}^{#1}}
\def\ov#1{\overline{#1}}

\newcommand\haus{\mathcal{H}}
\newcommand\Z{{\mathbb Z}}
\newcommand\N{{\mathbb N}}

\newcommand\res{\mathop{\hbox{\vrule height 7pt width .5pt depth 0pt
\vrule height .5pt width 6pt depth 0pt}}\nolimits}
\newcommand\supp{{\rm supp}\,}  
\newcommand\Inj{{\textrm{Inj}\,}}

\newcommand\F{{\mathcal{F}}}

\newcommand\An{{\mathcal{AN}}}

\newcommand\Is{{\mathfrak{Is}}}
\newcommand\gen{{\bf g}}

\begin{document}

\title[Genus bounds for minimal surfaces arising from 
min-max construction]
{Genus bounds for minimal surfaces arising from min-max constructions}

\author{Camillo De Lellis}%
\author{Filippo Pellandini}%
\address{Institut f\"ur Mathematik, Universit\"at Z\"urich, 
CH-8057 Z\"urich}

\email{camillo.delellis@math.uzh.ch and 
filippo.pellandini@math.uzh.ch}

\begin{abstract}
In this paper we prove genus bounds for closed embedded minimal 
surfaces in a closed $3$-dimensional manifold constructed via min-max 
arguments. A stronger estimate was announced by Pitts and 
Rubistein but 
to our knowledge its proof has never been published. 
Our proof follows ideas of Simon and uses an extension of
a famous result of Meeks, Simon and Yau on the convergence of minimizing 
sequences of isotopic surfaces. This result is proved in the second
part of the paper.

\end{abstract}

\maketitle
 
\tableofcontents

\section{Introduction}

\subsection{Min--max surfaces}
In \cite{CD} Tobias H. Colding and the second author
started a survey on  
constructing closed embedded minimal surfaces in a closed 
$3$-dimensional manifold 
via min--max arguments, including results 
of F. Smith, L. Simon, J. Pitts and H. Rubinstein. This paper
completes the survey by giving genus bounds 
for the final minmax surface.

The basic idea of min--max arguments over
sweep-outs goes back to Birkhoff, 
who used such a method to find simple closed
geodesics on spheres. In particular when $M^2$ is the $2$-dimensional
sphere we can find a $1$--parameter family of curves starting and ending
at a point curve in such a way that the induced map $F:\SS^2\to \SS^2$
has nonzero degree. 
Birkhoff's argument  (or the min-max
argument) allows us to conclude that $M$ has a nontrivial closed geodesic
of length less than or equal to the length of the longest curve in the
$1$-parameter family. A curve shortening argument gives that the
geodesic obtained in this way is simple.

Following \cite{CD} we introduce a suitable generalized setting
for sweepouts of $3$--manifolds by 
two--dimensional surfaces.
From now on, $M$, $\diff$ and $\Is$ will denote, respectively,
a closed $3$--dimensional Riemannian manifold,
the identity component 
of the diffeomorphism group of $M$, and 
the set
of smooth isotopies.  Thus $\Is$ consists of those maps 
\ $\psi\in C^\infty ([0,1]\times M, M)$ such that
$\psi(0, \cdot)$ is the identity and $\psi(t, \cdot)\in \diff$ 
for every $t$.

\begin{definition}\label{d:gensur}
A family 
$\{\Sigma_t\}_{t\in [0,1]}$ of surfaces of $M$ is said to be 
{\em continuous}\/ if 
\begin{itemize}
\item[(c1)] $\haus^2 (\Sigma_t)$ is 
a continuous function of $t$;
\item[(c2)] $\Sigma_t\to \Sigma_{t_0}$ in the
  Hausdorff topology whenever $t\to t_0$. 
\end{itemize}

A family $\{\Sigma_t\}_{t\in [0,1]}$ of subsets of $M$
is said to be a {\em generalized family}\/ of surfaces if there are a 
finite subset $T$ of $[0,1]$ and a finite set of points $P$ in $M$ such that\ 
\begin{itemize}
\item[1.] (c1) and (c2) hold;
\item[2.] $\Sigma_t$ is a surface for every $t\not \in T$;
\item[3.] For $t\in T$, $\Sigma_t$ is a surface in $M\setminus P$.
\end{itemize}

\end{definition}

With a small abuse of notation, we shall use the word ``surface''
even for the sets $\Sigma_t$ with $t\in T$.
To avoid confusion, families of surfaces will always be denoted by 
$\{\Sigma_t\}$. Thus, when referring to a surface a subscript 
will denote a real parameter, whereas a superscript will 
denote an integer as in a sequence.

Given a generalized family $\{\Sigma_t\}$ 
we can generate new generalized 
families via the following procedure.
Take an arbitrary map $\psi\in C^\infty 
([0,1]\times M, M)$ such that $\psi(t, \cdot)\in \diff$ for each $t$ and
define $\{\Sigma'_t\}$ by $\Sigma'_t=\psi (t, \Sigma_t)$.\label{i:proc}
We will say that a set $\Lambda$ of generalized families is 
{\em saturated}\/ if it is closed under this operation. 

\begin{remark}\label{r:P}
For technical reasons we require an additional property 
for any saturated set $\Lambda$ considered in this paper:
the existence of some $N=N(\Lambda)<\infty$ such that 
for any $\{\Sigma_t\}\subset \Lambda$, 
the set $P$ in Definition 
\ref{d:gensur} consists of at most $N$ points.  
\end{remark}

Given a family $\{\Sigma_t\}\in \Lambda$ we denote 
by $\F (\{\Sigma_t\})$
the area of its maximal slice and by $m_0 (\Lambda)$
the infimum of $\F$ taken over all families of $\Lambda$; that is, 
\begin{eqnarray}\label{e:min--max}
&&\F (\{\Sigma_t\}) = \max_{t\in [0,1]} \haus^2 (\Sigma_t) 
\qquad \text{ and }\\
&&m_0 (\Lambda) = \inf_{\Lambda} \F =
\inf_{\{\Sigma_t\}\in \Lambda}\, \left[ \max_{t\in [0,1]} \haus^2 
(\Sigma_t)\right]\,. 
\end{eqnarray}

If $\lim_n \F (\{\Sigma_t\}^n)=m_0 (\Lambda)$, then we say that 
the sequence of generalized families of surfaces 
$\{\{\Sigma_t\}^n\}\subset \Lambda$ is a 
{\em minimizing sequence}. Assume $\{\{\Sigma_t\}^n\}$ is a minimizing 
sequence and let $\{t_n\}$ 
be a sequence of parameters. If the areas of the slices $\{\Sigma^n_{t_n}\}$ converge 
to $m_0$, i.e. if $\haus^2 (\Sigma^n_{t_n})\to m_0 (\Lambda)$, 
then we say that $\{\Sigma^n_{t_n}\}$
is a {\em min--max sequence}.

An important point in the min--max construction is to find a saturated $\Lambda$ with $m_0 (\Lambda)>0$. 
For instance, this can be done by using the following 
elementary proposition proven in the Appendix of \cite{CD}.

\begin{propos}\label{p:morse}
Let $M$ be a closed $3$-manifold with a Riemannian metric and let 
$\{\Sigma_t\}$ be the level sets of a Morse function. The
smallest saturated set 
$\Lambda$ containing the family $\{\Sigma_t\}$
has $m_0 (\Lambda)>0$.
\end{propos}

The paper \cite{CD} reports a proof of the following 
regularity result.

\begin{theorem}\label{t:SS}[Simon--Smith] 
Let $M$ be a closed $3$-manifold with a Riemannian metric.  
For any saturated $\Lambda$, there is 
a min--max sequence $\Sigma^n_{t_n}$
converging in the sense of varifolds to a
smooth embedded minimal surface $\Sigma$ with area
$m_0 (\Lambda)$ (multiplicity is allowed).
\end{theorem}

\subsection{Genus bounds} In this note we bound the topology of 
$\Sigma$ under the assumption
that the $t$--dependence of $\{\Sigma_t\}$ 
is smoother than
just the continuity required in Definition \ref{d:gensur}. 
This is the content of the next definition.

\begin{definition}\label{d:gensur2}
A generalized family $\{\Sigma_t\}$ 
as in Definition \ref{d:gensur}
is said to be {\em smooth} if:
\begin{itemize}
\item[(s1)] $\Sigma_t$ varies smoothly in $t$ on 
$[0,1]\setminus T$;
\item[(s2)] For $t\in T$, $\Sigma_\tau\to \Sigma_t$ smoothly
in $M\setminus P$.
\end{itemize}
Here $P$ and $T$ are the sets of requirements 2. and 3. of
Definition \ref{d:gensur}. We assume further
that $\Sigma_t$ is orientable for any $t\not\in T$. 
\end{definition}

Note that, if a set $\Lambda$ consists of smooth
generalized families, then the elements of its saturation 
are still smooth generalized families. Therefore
the saturated set considered in Proposition \ref{p:morse}
is smooth.

We next introduce some notation which will be
consistently used during the proofs.
We decompose the surface $\Sigma$ of Theorem \ref{t:SS}
as $\sum_{i=1}^{N}n_i\Gamma^i$, 
where the $\Gamma^i$'s are the connected components
of $\Sigma$, counted without multiplicity, and 
$n_i\in \mathbb{N}\setminus \{0\}$ for every $i$.
We further divide the components $\{\Gamma^i\}$ into
two sets: the orientable ones, denoted by
$\mathcal{O}$, and the non--orientable ones,
denoted by $\mathcal{N}$. We are now ready
to state the main theorem of this paper.

\begin{theorem}\label{t:main}
Let $\Lambda$ be a saturated set of smooth generalized
families and $\Sigma$ and $\Sigma^n_{t_n}$ 
the surfaces produced in the proof of Theorem \ref{t:SS}
given in \cite{CD}. Then 
\begin{equation}\label{e:claim}
\sum_{\Gamma^i\in \mathcal{O}}\gen(\Gamma^i)
+ \frac{1}{2}\sum_{\Gamma^i\in \mathcal{N}} (\gen (\Gamma^i)-1)
\;\leq\; \gen_0 \;:=\; \liminf_{j\uparrow \infty} \liminf_{\tau\to t_j}
\gen (\Sigma^j_\tau)\, .
\end{equation}
\end{theorem}

\begin{remark}\label{r:epsilonj}
According to
our definition, $\Sigma^j_{t_j}$ 
is not necessarily a smooth submanifold, as
$t_j$ could be one of the exceptional parameters of point 3. in 
Definition \ref{d:gensur}. However, for each fixed
$j$ there is an $\eta>0$ such that $\Sigma^j_t$ is a smooth 
submanifold for every 
$t \in ]t_j-\eta, t_j[\cup ]t_j, t_j+\eta[$.
Hence the right hand side of \eqref{e:claim} makes sense.
\end{remark}

In fact the inequality \eqref{e:claim} holds with
$\gen_0 = \liminf_j \gen (\Sigma^j)$
for every limit $\Sigma$ of a sequence of surfaces $\Sigma^j$'s 
that enjoy certain
requirements of variational nature, i.e. that
are {\em almost minimizing in sufficiently small annuli}.
The precise statement will be given in
Theorem \ref{t:main2}, after introducing the 
suitable concepts. 

As usual, when $\Gamma$ is an orientable $2$--dimensional 
connected surface, its genus
$\gen (\Gamma)$ is defined as the number of handles that one
has to attach to a sphere in order to get a surface 
homeomorphic to $\Gamma$. When $\Gamma$ is non--orientable
and connected,
$\gen (\Gamma)$ is defined as the number of cross caps that
one has to attach to a sphere in order to get a surface 
homeomorphic to $\Gamma$ (therefore, if $\chi$ is the Euler
characteristic of the surface, then 
$$
\gen (\Gamma) \;=\left\{
\begin{array}{ll} 
\textstyle{\frac{1}{2}} (2-\chi) & \mbox{ if $\Gamma\in \mathcal{N}$}\\
2-\chi & \mbox{ if $\Gamma\in \mathcal{O}$}
\end{array}\right.
$$
see \cite{Massey}). 
For surfaces with more than one
connected component, the genus is simply the sum
of the genus of each connected component. 

Our genus estimate \eqref{e:claim} is weaker
than the one announced by Pitts and Rubinstein
in \cite{PR1}, which reads as follows
(cp. wih Theorem 1 and Theorem 2 in \cite{PR1}):
\begin{equation}\label{e:stronger}
\sum_{\Gamma^i\in \mathcal{O}} n_i \gen(\Gamma^i)
+ \frac{1}{2} \sum_{\Gamma^i\in \mathcal{N}} n_i \gen (\Gamma^i)
\;\leq\; \gen_0\, .
\end{equation}
In Section \ref{s:discuss} a very elementary example
shows that \eqref{e:stronger}
is false for sequences of almost minimizing surfaces
(in fact even for sequences which are locally strictly minimizing). 
In this case the correct estimate should be
\begin{equation}\label{e:correct}
\sum_{\Gamma^i\in \mathcal{O}} n_i \gen(\Gamma^i)
+ \frac{1}{2} \sum_{\Gamma^i\in \mathcal{N}} n_i (\gen (\Gamma^i)-1)
\;\leq\; \gen_0\, .
\end{equation}
Therefore, the improved estimate \eqref{e:stronger}
can be proved only by exploiting an argument of 
more global nature, using a more detailed analysis of the 
min--max construction.

The estimate \eqref{e:correct} 
respects the rough intuition that the approximating
surfaces $\Sigma^j$ are, after appropriate surgeries, isotopic
to coverings of the surfaces $\Gamma^i$. For instance $\Gamma$
can consist of a single component that is a real projective space,
and $\Sigma^j$ might be the boundary of a tubular neighborhood
of $\Gamma$ of size $\eps_j\downarrow 0$, i.e. a sphere. In this
case $\Sigma^j$ is a double cover of $\Gamma$.   

Our proof uses the ideas of an unpublished 
argument of Simon, reported by Smith in \cite{Sm} to show
the existence of an embedded minimal $2$--sphere
when $M$ is a $3$--sphere. These ideas do not
seem enough to show \eqref{e:stronger}: its proof probably 
requires a much more
careful analysis. In Section \ref{s:discuss} we discuss this issue.


\begin{remark}
The unpublished argument
of Simon has been used also by Gr\"uter and Jost in \cite{GJ}.
The core of Simon's argument is reported here with a 
technical simplification. We then give
a detailed proof of an auxiliary proposition which plays
a fundamental role in the argument. This part is, to our 
knowledge, new: neither Smith, nor Gr\"uter and Jost provide
a proof of it.
Smith suggests that the
proposition can be proved by suitably modifying the
arguments of \cite{MSY} and \cite{AS}. Though this is indeed
the case, the strategy suggested by Smith leads
to a difficulty which we overcome with a different
approach: see the discussion in Section \ref{s:MSY3}.
Moreover, \cite{Sm} does not discuss the ``convex--hull
property'' of Section \ref{s:MSY1},
which is a basic prerequisite to apply the
boundary regularity theory of Allard
in \cite{All2} (in fact we do not know of any boundary
regularity result in the minimal surface theory which
does not pass through some kind of convex hull
property). 
\end{remark}

\subsection{An example}
We end this introduction with a brief discussion of
how a sequence of closed surface $\Sigma^j$ could converge,
in the sense of varifolds, to a smooth surface
with higher genus. This example is
a model situation which must be ruled out by any
proof of a genus bound. First take
a sphere in $\RR^3$ and squeeze it in one direction
towards a double copy of a disk (recall that the 
convergence in the sense of varifolds does
not take into account the orientation). Next take the disk
and wrap it to form a torus in the standard way. With a 
standard diagonal
argument we find a sequence of smooth embedded 
spheres in $\RR^3$ which, in
the sense of varifolds, converges to a double copy of an
embedded torus. See Figure \ref{f:genusfails} below.

\begin{figure}[htbp]
\begin{center}
    \input{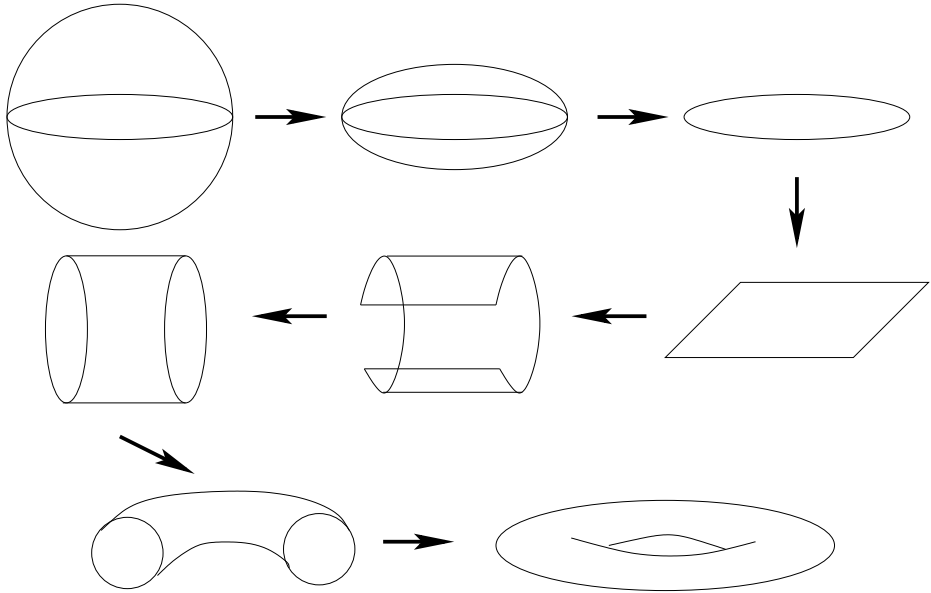}
    \caption{Failure of genus bounds under varifold
convergence. A sequence of embedded spheres converges
to a double copy of a torus.}
    \label{f:genusfails}
\end{center}
\end{figure}

This example does not occur in min--max sequences
for variational reasons. In particular, it follows
from the arguments of this paper that such a sequence
does not have the almost minimizing property in (sufficiently
small) annuli discussed
in Section \ref{s:prel}. 

\subsection{Plan of the paper} Section \ref{s:prel}
contains: some preliminaries on notational conventions,
a summary of the material of \cite{CD} used
in this note and the most precise statement of
the genus bounds (Theorem \ref{t:main2}). Section
\ref{s:overview} gives an overview of the proof of Theorem
\ref{t:main2}. In particular it reduces it to a statement
on lifting of paths, which we call Simon's Lifting Lemma
(see Proposition \ref{p:lifting}). Sections \ref{s:MSYbis}
and \ref{s:lifting} contain a proof of Simon's Lifting Lemma.
In Section \ref{s:MSYbis} we state a suitable modification
of a celebrated result of Meeks, Simon and Yau
(see \cite{MSY}) in which we handle minimizing sequences
of isotopic surfaces with boundaries (see Proposition
\ref{p:MSYbis}). 

Sections \ref{s:MSY1}, \ref{s:MSY2}, \ref{s:MSY3},
\ref{s:MSY4} and \ref{s:MSY5} show how to modify the theory of 
\cite{MSY} and \cite{AS} in order
to prove Proposition \ref{p:MSYbis}. 
Section \ref{s:MSY1}
discusses the convex--hull properties needed for the
boundary regularity.
In Section
\ref{s:MSY2} we introduce and prove the ``squeezing lemmas''
which allow to pass from almost--minimizing sequences
to minimizing sequences. Section \ref{s:MSY3} discusses
the $\gamma$--reduction and how one applies it
to get the interior regularity. We also point out why the
$\gamma$--reduction cannot be applied directly to 
the surfaces of Proposition \ref{p:MSYbis}.
Section \ref{s:MSY4} proves the boundary regularity.
Finally, section \ref{s:MSY5} handles the
part of Proposition \ref{p:MSYbis} involving limits of connected
components.

Section \ref{s:discuss} discusses 
the subtleties of the stronger
estimates \eqref{e:stronger} and \eqref{e:correct}. 


\section{Preliminaries and statement of the result}
\label{s:prel}

\subsection{Notation}
Throughout this paper our notation will be consistent with
the one of \cite{CD}, explained in Section 2 of that paper.
For the reader's convenience we recall some of these
conventions in the following table.

\begin{tabular}{lll}
$T_xM$ && the tangent space of $M$ at $x$\\
$TM$ && the tangent bundle of $M$.\\
$\Inj(M)$ && the injectivity radius of $M$.\\
$\haus^2$ && the $2$--d Hausdorff 
measure in the metric space $(M,d)$.\\
$\haus^2_e$ && the $2$--d Hausdorff 
measure in the euclidean space $\rn{3}$.\\
$B_\rho (x)$ && open ball\\
$\ov{B}_\rho (x)$ && closed ball\\
$\partial B_\rho (x)$ && distance sphere of radius $\rho$ in $M$.\\
$\diam (G)$ && diameter of a subset $G\subset M$.\\
$d(G_1, G_2)$ && the Hausdorff distance between the subsets\\
&& $G_1$ and $G_2$ of $M$.\\
$\D$, $\D_\rho$ && the unit disk and the disk of radius $\rho$ in $\rn{2}$.\\
$\B$, $\B_\rho$ && the unit ball and the ball of radius $\rho$ in $\rn{3}$.\\
$\exp_x$ && the exponential map in $M$ at $x\in M$.\\
$\Is (U)$ && smooth isotopies which leave 
$M\setminus U$ fixed.\\
$G^2(U)$, $G(U)$ && grassmannian of (unoriented) 
$2$--planes on $U\subset M$.\\
$\an (x,\tau, t)$ && the open annulus 
$B_t (x)\setminus \overline{B}_\tau (x)$.\\
$\An_r (x)$ && the set 
$\{\an (x, \tau, t) \mbox{ where $0<\tau<t<r$}\}$.\\
$C^\infty (X,Y)$ && smooth maps from $X$ to $Y$.\\
$C^\infty_c (X,Y)$ && smooth maps with compact support
from $X$\\
&& to the vector space $Y$.
\end{tabular}

\subsection{Varifolds} We will need to recall some basic facts from the 
theory of varifolds; see for instance  
chapter 4 and chapter 8 of \cite{Si} for further information.
Varifolds are a convenient way of generalizing surfaces to a category that 
has good compactness properties. An advantage of varifolds, over other
generalizations (like currents), is that they do not allow for
cancellation of mass. This last property is fundamental for the
min--max construction.

If $U$ is an open 
subset of $M$, any finite nonnegative measure on the Grassmannian of 
unoriented $2$--planes on $U$ is said to be a {\em $2$--varifold in $U$}. 
The Grassmannian 
of $2$--planes will be denoted by $G^2(U)$ and the vector space of 
$2$--varifolds is denoted by $\mathcal{V}^2 (U)$.  
Throughout we will consider only $2$--varifolds; thus we drop the 2.

We endow $\mathcal{V} (U)$ with the topology of
the weak convergence in the sense of
measures, thus we say that a sequence $V^k$ of varifolds converge to
a varifold $V$ if for every function $\varphi\in C_c (G(U))$ 
$$
\lim_{k\to \infty} \int \varphi (x, \pi)\, dV^k (x, \pi)
\;=\; \int \varphi (x, \pi)\, dV (x, \pi)\, .
$$
Here $\pi$ denotes a $2$--plane of $T_x M$.
If $U'\subset U$ and $V\in \mathcal{V} (U)$, then we denote by 
$V\res U'$ the restriction of the measure $V$ to $G (U')$. Moreover, 
$\|V\|$ will be the unique measure on $U$ satisfying
$$
\int_U \varphi (x) \,d\|V\| (x)\;=\;
\int_{G(U)} \varphi (x) \,dV (x, \pi)\qquad \forall \varphi\in C_c
(U)\, .
$$
The support of $\| V\|$, denoted by $\supp (\|V\|)$, is the 
smallest closed set outside
which $\|V\|$ vanishes identically.
The number $\|V\|(U)$ will be
called the {\em mass of $V$ in $U$}. When $U$ is clear from the context, 
we say briefly the {\em mass of $V$}.

Recall also that a $2$--dimensional rectifiable set is a countable union
of closed subsets of $C^1$ surfaces (modulo sets of $\haus^2$--measure 0).
Thus, if $R\subset U$ is a $2$--dimensional rectifiable set 
and $h:R\to \rn{+}$ is a Borel function, then we can define a 
varifold $V$ by 
\begin{equation}\label{e:defvar}
\int_{G (U)} \varphi (x, \pi) \,dV (x, \pi)=
\int_R h(x) \varphi (x, T_x R) \,d\haus^2 (x)\, \quad \forall
\varphi\in C_c (G (U))\, .
\end{equation}
Here $T_x R$ denotes the tangent plane to $R$ in $x$.
If $h$ is integer--valued, then we say that $V$ is an 
{\em integer rectifiable varifold}.
If $\Sigma=\bigcup n_i \Sigma_i$, then 
by slight abuse of notation we use $\Sigma$ for the 
varifold induced by $\Sigma$ via \eqref{e:defvar}.

\subsection{Pushforward, first variation, monotonicity formula}
If $V$ is a varifold induced by a surface 
$\Sigma\subset U$ and $\psi:U\to U'$ a diffeomorphism, 
then we let $\psi_\# V\in \mathcal{V} (U')$ 
be the varifold induced by the 
surface $\psi (\Sigma)$. The definition of $\psi_\# V$ can be naturally 
extended to {\em any} $V\in \mathcal{V} (U)$ by
$$
\int \varphi(y, \sigma)\, d(\psi_\# V) (y, \sigma)
\;=\; \int J \psi (x, \pi)\, \varphi 
(\psi (x), d\psi_x (\pi))\, dV (x, \pi)\, ;
$$
where $J \psi (x, \pi)$ denotes the Jacobian determinant (i.e. the area
element) of the differential $d\psi_x$ restricted to the plane $\pi$;
cf. equation (39.1) of \cite{Si}.

Given a smooth vector field $\chi$, let $\psi$ be the isotopy
generated by $\chi$, i.e. with ${\textstyle \frac{\partial
    \psi}{\partial t} = \chi( \psi)}$. The   
first variation of $V$ with respect to $\chi$ is
defined as
$$
[\delta V] (\chi) \;=\; \left. \frac{d}{dt} (\|\psi (t, \cdot)_\# V\|)
\right|_{t=0}\, ;
$$
cf. sections 16 and 39 of \cite{Si}. When $\Sigma$ is a smooth surface 
we recover the classical definition of first variation of a
surface:
$$
[\delta \Sigma] (\chi) \;=\; \int_\Sigma {\rm div}_{\Sigma} \chi\, d\haus^2 
\;=\; \left. \frac{d}{dt} (\haus^2 (\psi(t,\Sigma)))\right|_{t=0}\, .
$$
If $[\delta V] (\chi)=0$ for every $\chi\in C^\infty_c (U,TU)$, then $V$ 
is said to be {\em stationary in $U$}. Thus stationary varifolds are 
natural generalizations of minimal surfaces.

Stationary varifolds in Euclidean spaces satisfy the monotonicity
formula (see sections 17 and 40 of \cite{Si}):
\begin{equation}\label{e:MonFor1}
\mbox{For every $x$ the function } 
f(\rho)= \frac{\|V\| (B_\rho (x))}{\pi \rho^2}
\mbox{ is non--decreasing.} 
\end{equation}
When $V$ is a stationary varifold in a Riemannian manifold a similar
formula with an error term holds. Namely, there exists a constant
$C (r)\geq 1$ such that
 \begin{equation}\label{e:MonFor}
f(s)\;\leq\; C(r) f(\rho) \qquad \mbox{whenever $0<s<\rho<r$.}
\end{equation}
Moreover, the constant $C(r)$ approaches $1$ as 
$r\downarrow 0$. This property allows us to define the
{\em density} of a stationary varifold $V$ at $x$, by
$$
\theta (x, V)\;=\; \lim_{r\downarrow 0} \frac{\|V\| (B_r (x))}{\pi r^2}.
$$  
Thus $\theta (x, V)$ corresponds to the upper
density $\theta^{*2}$ of the measure $\|V\|$ as defined in section 3
of \cite{Si}.
 
\subsection{Curvature estimates for stable minimal surfaces}
\label{ss:curvest}
In many of the proofs we will use
Schoen's curvature estimate (see \cite{Sc}) 
for stable minimal surfaces.  Recall that this estimate asserts that,
if $U\subset\subset M$, then there exists a universal
  constant, $C(U)$, such that for every stable minimal surface 
$\Sigma\subset U$ 
with $\partial \Sigma\subset \partial U$ and second fundamental form $A$ 
\begin{equation}\label{e:curva1}
|A|^2 (x) \; \leq\; \frac{C(U)}{d^2 (x, \partial U)}\, \qquad \forall x\in
\Sigma\, .
\end{equation}
In fact, what we will use is not the actual curvature estimate, rather 
it is the following consequence of it:
\begin{eqnarray}
&\mbox{If $\{\Sigma^n\}$ is a sequence of stable
minimal surfaces in $U$, then a}&\nonumber\\
&\mbox{subsequence converges 
to a stable minimal surface $\Sigma^\infty$}\, .&\label{e:curvaclaim}
\end{eqnarray}

\subsection{Almost minimizing min--max sequences}
Next, we assume that $\Lambda$ is a fixed saturated set
and we begin by recalling the building blocks of the proof of
Theorem \ref{t:SS}. First of all, in \cite{CD}, following ideas
of Pitts and Almgren (see \cite{P} and \cite{Alm}), 
the authors reported a proof of the following proposition
(cp. with Proposition 3.1 in \cite{CD}).

\begin{propos}\label{p:goodbis} There exists a minimizing sequence 
$\{\{\Sigma_t\}^n\}\subset \Lambda$ 
such that\ every min--max sequence $\{\Sigma^n_{t_n}\}$ 
clusters to stationary 
varifolds.
\end{propos}

It is well--known that stationary varifolds are not, 
in general, 
smooth minimal surfaces. 
The regularity theory of Theorem \ref{t:SS}
relies on the definition of almost minimizing sequence, a concept
introduced by Pitts in \cite{P} and based on ideas of Almgren
(see \cite{Alm}). Roughly speaking
a surface $\Sigma$ is almost minimizing if 
any path of surfaces $\{\Sigma_t\}_{t\in
  [0,1]}$ starting at $\Sigma$ and such 
that $\Sigma_1$ has small area 
(compared to $\Sigma$) must necessarily 
pass through a surface with large area.
Our actual definition, following Smith and Simon,
is in fact more restrictive: we will require
the property above only for families $\{\Sigma_t\}$ given by 
smooth isotopies. 

\begin{definition}\label{AM1} Given $\eps>0$, 
  an open set $U\subset M^3$, and a
  surface $\Sigma$, we say that $\Sigma$ is {\em $\eps$--a.m.\ in $U$}\/ if
  there {\sc does not} exist any isotopy $\psi$ supported in $U$ such that\
  \begin{eqnarray}
  &&\mbox{$\haus^2 (\psi (t,\Sigma))\leq \haus^2 (\Sigma)+\eps/8$ for all
  $t$;}\label{AM(a)}\\ 
  &&\mbox{$\haus^2 (\psi (1,\Sigma))\leq \haus^2 (\Sigma)-\eps$.}\label{AM(b)}
  \end{eqnarray}
\end{definition}

Using a combinatorial argument due to Almgren 
and exploited by Pitts in \cite{P}, 
the second step of \cite{CD} was to show
Proposition \ref{p:existbis} below.

\begin{remark}
In fact, the statement of Proposition
\ref{p:existbis} does not coincide exactly with 
the corresponding Proposition 5.1 of \cite{CD}. However, it is easy
to see that Proposition 5.3 of \cite{CD} yields
the slightly small precise statement given below.
\end{remark}

\begin{propos}\label{p:existbis}
There exists a function $r:M\to \rn{+}$ and a min--max sequence
$\Sigma^j=\Sigma^j_{t_j}$ such that:
\begin{itemize}
\item in every annulus 
$\an$ centered at $x$ and with outer radius at most $r(x)$, 
$\Sigma^j$ is $1/j$--a.m. provided $j$ is large enough;
\item In any such annulus, $\Sigma^j$ is smooth when $j$ 
is sufficiently large;
\item $\Sigma^j$ converges to a stationary varifold $V$ in $M$, 
as $j\uparrow \infty$.
\end{itemize}
\end{propos}

The following Theorem completed the proof of Theorem \ref{t:SS}
(cp. with Theorem 7.1 in \cite{CD}).

\begin{theorem}\label{t:regularity}
Let $\{\Sigma^j\}$ be a sequence of surfaces in
$M$ and assume the existence of a function $r:M\to \rn{+}$ 
such that
the conclusions of Proposition \ref{p:existbis} hold. 
Then $V$ is a smooth minimal surface. 
\end{theorem}

The proof of this Theorem draws heavily on a fundamental result
of Meeks, Simon and Yau (\cite{MSY}). A suitable version of
it plays a fundamental role also in this paper and since 
the modifications of the ideas of \cite{MSY}
needed in our case are complicated, we will discuss them
later in detail. From now on, in order to simplify our notation,
a sequence $\{\Sigma^j\}$ satisfying the conclusions of 
Proposition \ref{p:existbis}
will be simply called {\em 
almost minimizing in sufficiently small annuli}.

\subsection{Statement of the result}
Our genus estimate is valid, in general, for limits
of sequences of surfaces which are almost minimizing in 
sufficiently small annuli.

\begin{theorem}\label{t:main2} Let $\Sigma^j = \Sigma^j_{t_j}$
be a sequence which is a.m. in sufficiently small annuli. 
Let $V = \sum_i n_i \Gamma^i$ be the varifold limit of 
$\{\Sigma^j\}$, 
where $\Gamma^i$ are as in Theorem \ref{t:main}. Then
\begin{equation}\label{e:claim2}
\sum_{\Gamma^i\in \mathcal{O}}\gen(\Gamma^i)
+ \frac{1}{2} \sum_{\Gamma^i\in \mathcal{N}} (\gen (\Gamma^i)-1)
\;\leq\; \liminf_{j\uparrow \infty}
\liminf_{\tau \to t_j} \gen (\Sigma^j_{\tau})\, .
\end{equation}
\end{theorem}

\section{Overview of the proof}\label{s:overview}

In this section we give an overview of the proof of
Theorem \ref{t:main2}. Therefore we fix a min--max sequence
$\Sigma^j = \Sigma^j_{t_j}$ as in Theorem \ref{t:main2}
and we let $\sum_i n_i \Gamma^i$ be its varifold limit. 
Consider the smooth surface
$\Gamma = \cup_i \Gamma^i$ and let $\eps_0>0$ be so small
that there exists a smooth retraction
of the tubular neighborhood $T_{2\eps_0} \Gamma$ onto $\Gamma$.
This means that, for every $\delta<2\eps_0$,
\begin{itemize}
\item $T_{\delta} \Gamma^i$
are smooth open sets with pairwise disjoint closures;
\item if $\Gamma^i$ is orientable, then $T_\delta
\Gamma^i$ is diffeomorphic to $\Gamma^i\times ]- 1,1[$; 
\item if $\Gamma^i$ is non--orientable,
then the boundary of $T_\delta \Gamma^i$ is 
an orientable double cover of $\Gamma^i$. 
\end{itemize}

\subsection{Simon's Lifting Lemma}
The following Proposition
is the core of the genus bounds. Similar
statements have been already used in the literature
(see for instance \cite{GJ} and \cite{FH}). We recall that 
the surface $\Sigma^j$ might not be everywhere
regular, and we denote by $P_j$ its set
of singular points (possibly empty).

\begin{propos}[Simon's Lifting Lemma]\label{p:lifting}
Let $\gamma$ be a closed simple curve on $\Gamma^i$ and let
$\eps\leq \eps_0$ be positive. Then, for $j$ large enough, there
is a positive $n\leq n_i$ and a closed curve 
$\tilde{\gamma}^j$ on $\Sigma^j\cap T_\eps
\Gamma^i\setminus P_j$ which is homotopic
to $n\gamma$ in $T_\eps \Gamma^i$.
\end{propos}

Simon's lifting Lemma implies directly the genus bounds
if we use the characterization of homology groups through 
integer rectifiable currents and some more geometric measure
theory. 
However, we choose to conclude the proof in a more elementary
way, using Proposition \ref{p:intorno} below. 

\subsection{Surgery}
The idea is that, for $j$ large enough, one can modify
any $\{\Sigma^j_t\}$ sufficiently close to $\Sigma^j =
\Sigma^j_{t_j}$ through surgery to a new surface
$\tilde{\Sigma}^j_t$ such that
\begin{itemize}
\item the new surface lies in a tubular neighborhood of $\Gamma$;
\item it coincides with the old surface in a yet smaller
tubular neighborhood.
\end{itemize}
 The surjeries that we will use in this paper are of two 
kind: we are allowed to
\begin{itemize}
\item remove a small cylinder and replace it
by two disks (as in Fig. \ref{f:cutting});
\item discard a connected component.
\end{itemize}
We give below the precise definition.

\begin{figure}[htbp]
\begin{center}
    \input{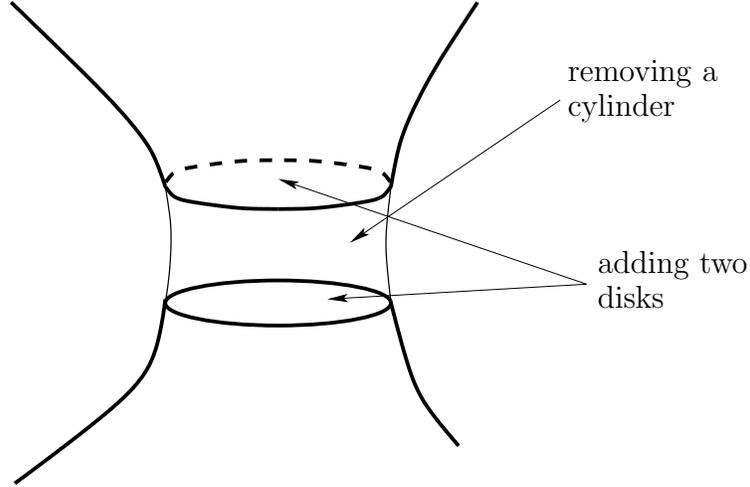}
    \caption{Cutting away a neck}
    \label{f:cutting}
\end{center}
\end{figure}

\begin{definition}\label{d:surgery}
Let $\Sigma$ and $\tilde{\Sigma}$ be two closed smooth
embedded surfaces. We say that $\tilde{\Sigma}$ is obtained
from $\Sigma$ by cutting away a neck if:
\begin{itemize}
\item $\Sigma\setminus \tilde{\Sigma}$ is 
homeomorphic to $S^1\times ]0,1[$;
\item $\tilde{\Sigma}\setminus \Sigma$ is 
homeomorphic to the disjoint union of two open
disks;
\item $\tilde{\Sigma}\Delta \Sigma$ is a contractible sphere.
\end{itemize}
We say that $\tilde{\Sigma}$ is obtained from $\Sigma$
through surgery if there is a finite number of surfaces
$\Sigma_0=\Sigma, \Sigma_1, \ldots, \Sigma_N = 
\tilde{\Sigma}$ such that each $\Sigma_k$ is 
\begin{itemize}
\item either isotopic to the union of 
some connected components of $\Sigma_{k-1}$;
\item or obtained from $\Sigma_{k-1}$
by cutting away a neck.
\end{itemize}
\end{definition}

Clearly, if $\tilde{\Sigma}$ is obtained from
$\Sigma$ through surgery, then $\gen (\tilde{\Sigma})
\leq \gen (\Sigma)$. 
We are now ready to state our next Proposition.

\begin{propos}\label{p:intorno}
Let $\eps\leq \eps_0$ be positive.
For each $j$ sufficiently large and for
$t$ sufficiently close to $t_j$, we can find a surface
$\tilde{\Sigma}^j_t$ obtained from $\Sigma^j_t$ through surgery
and satisfying the following properties:
\begin{itemize}
\item $\tilde{\Sigma}^j_t$ is contained in $T_{2\eps} \Gamma$;
\item $\tilde{\Sigma}^j_t\cap T_\eps \Gamma =
\Sigma^j_t\cap T_\eps \Gamma$.
\end{itemize}
\end{propos}

\subsection{Proof of Theorem
\ref{t:main2}}
Proposition \ref{p:intorno} and Proposition
\ref{p:lifting} allow us to conclude the proof of
Theorem \ref{t:main2}. We only need the following
standard fact for the first integral homology group of
a smooth closed connected surface
(see Sections 4.2 and 4.5 of \cite{Massey}).

\begin{lemma}\label{l:algebra}
Let $\Gamma$ be a connected closed $2$--dimensional
surface with genus $\gen$.
If $\Gamma$ is orientable, then $H^1 (\Gamma) = \Z^{2\gen}$.
If $\Gamma$ is non--orientable, then $H^1 (\Gamma) = 
\Z^{\gen -1}\times
\Z_2$.
\end{lemma}

The proof of Proposition \ref{p:intorno} is given below,
at the end of this section. 
The rest of the
paper is then dedicated to prove Simon's Lifting Lemma. 
We now come to the proof of Theorem \ref{t:main2}.

\begin{proof}[Proof of Theorem \ref{t:main2}]
Define $m_i = \gen (\Gamma^i)$ if $i$ is orientable
and $(\gen (\Gamma^i)-1)/2$ if not. Our aim is to show
that 
\begin{equation}\label{e:claim3}
\sum_i m_i \;\leq\; \liminf_{j\uparrow \infty}
\liminf_{t \to t_j} \gen (\Sigma^j_t)\, .
\end{equation}
By Lemma \ref{l:algebra}, for each $\Gamma^i$ there are $2m_i$ curves
$\gamma^{i,1},\ldots, \gamma^{i, 2m_i}$ with the following property:
\begin{itemize}
\item[(Hom)] If $k_1, \ldots, k_{2m_i}$ are integers such that
$k_1 \gamma^{i,1} + \ldots + k_{2m_i} \gamma^{i,2m_i}$ 
is homologically trivial in $\Gamma^i$, then $k_l=0$ for every $l$.
\end{itemize}
Since $\eps<\eps_0/2$, $T_{2\eps} \Gamma^i$ can be retracted smoothly
on $\Gamma^i$. Hence:
\begin{itemize}
\item[(Hom')] If $k_1, \ldots, k_{2m_i}$ are integers such that
$k_1 \gamma^{i,1} + \ldots + k_{2m_i} \gamma^{i,2m_i}$ is 
homologically trivial in $T_{2\eps} \Gamma^i$, 
then $k_l=0$ for every $l$.
\end{itemize}

Next, fix $\eps< \eps_0$
and let $N$ be sufficiently large so that,
for each $j\geq N$, Simon's Lifting Lemma applies to each curve
$\gamma^{i,l}$. We require, moreover, that $N$ is large enough so
that Proposition \ref{p:intorno} applies to every $j>N$.

Choose next any $j>N$ and consider
the curves $\tilde{\gamma}^{i,l}$ lying in
$T_\eps \Gamma\cap \Sigma^j$ given by Simon's Lifting Lemma. 
Such surfaces are therefore homotopic to 
$n_{i,l} \gamma^{i,l}$ in $T_\eps
\Gamma^i$, where each $n_{i,l}$ is a positive integer.
Moreover, for each $t$ sufficiently
close to $t_j$ consider the surface $\tilde{\Sigma}^j_t$
given by Proposition \ref{p:intorno}.
The surface $\tilde{\Sigma}^j_t$
decomposes into the finite number of components
(not necessarily connected) $\tilde{\Sigma}^j_t\cap T_{2\eps}
\Gamma^i$. Each such surface is orientable and
\begin{equation}
\sum_i \gen (\tilde{\Sigma}^j_t\cap T_{2\eps}
\Gamma^i) \;=\;
\gen (\tilde{\Sigma}^j_t)
\;\leq\; \gen (\Sigma^j_t)\, .
\end{equation}
We claim that
\begin{equation}\label{e:claim10}
m_i \;\leq\; \liminf_{t\to t_j}
\gen (\tilde{\Sigma}^j_t\cap T_{2\eps} \Gamma^i)\, ,
\end{equation}
which clearly would conclude the proof.

Since $\Sigma^j_t$ converges smoothly to
$\Sigma^j$ outside $P_j$, we conclude that
$\tilde{\Sigma}^j_t\cap T_\eps
\Gamma^i$ converges smoothly to
$\Sigma^j\cap T_\eps \Gamma^i$ outside $P_j$.
Since each $\gamma^{i,l}$ does not intersect $P_j$,
it follows that, for $t$ large enough,
there exist curves $\hat{\gamma}^{i,l}$ contained
in $\tilde{\Sigma}^j_t\cap T_\eps \Gamma^i$ and homotopic  
to $\tilde{\gamma}^{i,l}$ in $T_{\eps} \Gamma^i$.

\medskip

Summarizing:
\begin{itemize}
\item[(i)] Each $\tilde{\gamma}^{i,l}$ is homotopic
to $n_{i,l} \gamma^{i,l}$ in $T_{2\eps} \Gamma^i$ for
some positive integer $n_{i,l}$;
\item[(ii)] Each $\tilde{\gamma}^{i,l}$
is contained in $\tilde{\Sigma}^j_t\cap T_{2\eps}
\Gamma^i$;
\item[(iii)] $\tilde{\Sigma}^j_t\cap T_{2\eps}
\Gamma^i$ is a closed surface;
\item[(iv)] If $c_1 \gamma^{i,1} + \ldots +
c_{2m_i} \gamma^{i, 2m_i}$ is homologically trivial
in $T_{2\eps}
\Gamma^i$ and the $c_l$'s are integers, then they are all $0$.
\end{itemize}
These statements imply that:
\begin{itemize}
\item[(Hom'')] If $c_1 \tilde{\gamma}^{i,1} + \ldots +
c_{2m_i} \tilde{\gamma}^{i, 2m_i}$ is homologically
trivial in $\tilde{\Sigma}^j_t\cap T_{2\eps} \Gamma^i$
and the $c_l$'s are integers, then they are all $0$.
\end{itemize}
From Lemma \ref{l:algebra},
we conclude again that $\gen (\tilde{\Sigma}^j_t\cap T_{2\eps}
\Gamma^i)\geq m_i$. 
\end{proof}

\subsection{Proof of Proposition \ref{p:intorno}}
Consider the set $\Omega = T_{2\eps} \Gamma \setminus 
\overline{T_\eps \Gamma}$. Since $\Sigma^j$ converges,
in the sense of varifolds, to $\Gamma$, we have
\begin{equation}\label{e:svanisce}
\lim_{j\uparrow \infty}
\limsup_{t\to t_j} \haus^2 (\Sigma^j_t\cap \Omega) \;=\; 0\, .
\end{equation}
Let $\eta>0$ be a positive number to be fixed later and
consider $N$ such that 
\begin{equation}\label{e:svanisce2}
\limsup_{t\to t_j} \haus^2 (\Sigma^j_t\cap \Omega) \;<\; \eta/2\, 
\qquad \mbox{for each $j\geq N$.}
\end{equation}
Fix $j\geq N$ and let $\delta_j>0$ be such that 
\begin{equation}\label{e:svanisce3}
\haus^2 (\Sigma^j_t\cap \Omega) \;<\; \eta\, 
\qquad \mbox{if $|t_j-t|< \delta_j$.}
\end{equation}
For each $\sigma\in ]\eps, 2\eps[$ consider $\Delta_\sigma :=
\partial\, (T_\sigma \Gamma)$, i.e. the boundary
of the tubular neighborhood $T_\sigma \Gamma$. 
The surfaces $\Delta_\sigma$ are a smooth
foliation of $\Omega\setminus \Gamma$ and therefore, 
by the coarea formula
\begin{equation}\label{e:svanisce4}
\int_\eps^{2\eps} \Length (\Sigma^j_t\cap \Delta_\sigma)\, 
d\sigma \;\leq\; C \haus^2 (\Sigma^j_t\cap \Omega) \;<\; C \eta\, 
\end{equation}
where $C$ is a constant independent of $t$ and $j$. Therefore, 
\begin{equation}\label{e:svanisce5}
\Length (\Sigma^j_t\cap \Delta_\sigma) < \frac{2C\eta}{\eps}\, 
\end{equation}
holds for a set of $\sigma$'s with measure at least $\eps/2$.

By Sard's Lemma we can fix a $\sigma$ such that
\eqref{e:svanisce4} holds and $\Sigma^j_t$ intersects 
$\Delta_t$ transversally. 

For positive constants $\lambda$ and $C$, 
independent of $j$ and $t$,
the following holds: 
\begin{itemize}
\item[(B)] For any $s\in ]0, 2\eps[$,
any simple closed curve $\gamma$ lying on 
$\Delta_s$ with $\Length (\gamma)\leq 
\lambda$ bounds an embedded disk $D\subset \Delta_s$ with 
$\diam (D)\leq C \Length (\gamma)$.
\end{itemize}

Assume that $2C\eta/\eps < \lambda$. By construction,
$\Sigma^j_t\cap \Delta_\sigma$ is a finite collection
of simple curves. Consider $\tilde{\Omega} := T_{\sigma+\delta}
\Gamma \setminus \overline{T_{\sigma-\delta} \Gamma}$.
For $\delta$ sufficiently small, $\tilde{\Omega}\cap
\Sigma^j_t$ is a finite collection of cylinders,
with upper bases lying on $\Delta_{\sigma+\delta}$ and lower
bases lying on $\Delta_{\sigma-\delta}$. We ``cut away''
this finite number of necks by removing
$\tilde{\Omega}\cap
\Sigma^j_t$ and replacing them with the two disks
lying on $\Delta_{\sigma-\delta}\cup \Delta_{\sigma+\delta}$
and enjoying the bound (B). For a suitable choice of
$\eta$, the union of each neck and of the corresponding 
two disks has sufficiently small diameter. This surface
is therefore a compressible sphere, which implies
that the new surface $\hat{\Sigma}^j_t$ is obtained from
$\Sigma^j_t$ through surgery. 

We can smooth it a little: the smoothed surface will still
be obtained from $\Sigma^j_t$ through surgery and will
not intersect $\Delta_\sigma$. Therefore 
$\tilde{\Sigma}^j_t := \hat{\Sigma}^j_t
\cap T_\sigma \Gamma$ is a closed surface and is obtained
from $\hat{\Sigma}^j_t$ by dropping a finite number of connected
components.
\qed

\section{Proof of Proposition \ref{p:lifting}. Part I:
Minimizing sequences of isotopic surfaces}
\label{s:MSYbis}

A key point in the proof of Simon's Lifting Lemma is
Proposition \ref{p:MSYbis} below. Its proof,
postponed to later sections,
relies on the
techniques introduced by Almgren and Simon in 
\cite{AS} and Meeks, Simon and Yau in \cite{MSY}. 
Before stating the proposition
we need to introduce some notation.

\subsection{Minimizing sequences of isotopic
surfaces} \begin{definition}\label{minprinciple}
Let $\mathcal{I}$ 
be a class of isotopies of $M$ and $\Sigma\subset M$ 
a smooth embedded
surface. If $\{\varphi^k\}\subset \mathcal{I}$ and 
$$
\lim_{k\to\infty} 
\haus^2 (\varphi^k (1,\Sigma))\;=\;\inf_{\psi\in \mathcal{I}} 
\haus^2 (\psi (1,\Sigma))\, ,
$$
then we say that $\varphi^k(1,\Sigma)$ is a {\em minimizing
  sequence for Problem $(\Sigma, \mathcal{I})$}.

If $U$ is an open set of $M$, $\Sigma$
a surface with $\partial \Sigma \subset \partial U$
and $j\in \N$ an integer, then we define
\begin{equation}\label{e:defIsj}
\Is_j (U,\Sigma)\;:=\;
\left\{\psi\in \Is (U)\big|\, \, 
\haus^2 (\psi (\tau, \Sigma))
\leq \haus^2 (\Sigma)+1/(8j)\quad \forall 
\tau\in [0,1]\right\}\, .
\end{equation}
\end{definition}

\begin{propos}\label{p:MSYbis}
Let $U\subset M$ be an open ball with sufficiently small
radius and consider a smooth embedded
surface $\Sigma$ such that $\partial \Sigma\subset
\partial U$ is also smooth. Let $\Delta^k:= \varphi^k (1, \Sigma)$ be
a minimizing sequence for Problem $(\Sigma, \Is_j (U,
\Sigma))$, converging to a stationary varifold $V$. 
Then, $V$ is 
a smooth minimal surface $\Delta$ with smooth boundary
$\partial \Delta = \partial \Sigma$. 

Moreover, if we form a new sequence $\tilde{\Delta}^k$
by taking an arbitrary union of 
connected components of $\Delta^k$, it converges,
up to subsequences, to the union of some connected components
of $\Delta$.
\end{propos}

In fact, we believe that the proof of
Proposition \ref{p:MSYbis} could be modified to
include any open set $U$ with smooth, uniformly convex boundary.
However, such a statement would imply several 
technical complications in Section \ref{s:MSY1} and
hence goes beyond our scopes. Instead, the following
simpler statement can be proved directly with our arguments, though
we do not give the details.

\begin{propos}\label{p:MSYtris}
Let $U\subset M$ be a uniformly convex open set
with smooth boundary and consider a smooth embedded
surface $\Sigma$ such that $\partial \Sigma\subset
\partial U$ is also smooth. Let $\Delta^k:= \varphi^k (1, \Sigma)$ be
a minimizing sequence for Problem $(\Sigma, \Is (U))$, 
converging to a stationary varifold $V$. 
Then, $V$ is 
a smooth minimal surface $\Delta$ with smooth boundary
$\partial \Delta = \partial \Sigma$. 

Moreover, if we form a new sequence $\tilde{\Delta}^k$
by taking an arbitrary union of 
connected components of $\Delta^k$, it converges,
up to subsequences, to the union of some connected components
of $\Delta$.
\end{propos}

\subsection{Elementary remarks on minimizing surfaces}
We end this section by collecting some properties
of minimizing sequences of isotopic surfaces
which will be used often throughout this paper.
We start with two very elementary remarks.

\begin{remark}\label{r:contained}
If $\Sigma$ is $1/j$--a.m. in an open set 
$U$ and $\tilde{U}$ is an open set contained in $U$,
then $\Sigma$ is $1/j$--a.m. in $\tilde{U}$.
\end{remark}

\begin{remark}\label{r:stillam}
If $\Sigma$ is $1/j$--a.m. in $U$ and $\psi\in \Is_j (\Sigma, U)$
is such that $\haus^2 (\psi (1, \Sigma))\leq \haus^2 (\Sigma)$,
then $\psi (1, \Sigma)$ is $1/j$--a.m. in $U$. 
\end{remark}

Next we collect two lemmas. Their proofs are
short and we include them below for the reader's convenience.

\begin{lemma}\label{l:stillam}
Let $\Sigma_j$ be $1/j$--a.m. in annuli and $r: M \to
\rn{+}$ be the function of Theorem \ref{t:regularity}. Assume
$U$ is an open set with closure contained in 
$\an (x,\tau, \sigma)$, where $\sigma<r(x)$.
Let $\psi_j\in \Is_j (\Sigma_j, U)$ be such that 
$\haus^2 (\psi_j (1, \Sigma_j))\leq \haus^2 (\Sigma)$.
Then $\psi_j (1, \Sigma_j)$ is $1/j$--a.m. in sufficiently
small annuli.
\end{lemma}

\begin{proof} Recall the definition of $1/j$--a.m. in sufficiently
small annuli.
This means that there is a function $r: M\to \rn{+}$ such that
$\Sigma$ is $1/j$--a.m. on every annulus centered
at $y$ and with outer radius smaller than $r (y)$.
Let $\an (x, \tau, \sigma)$ be an annulus on which $\Sigma$
is $1/j$--a.m. and $U\subset\subset \an (x, \tau, \sigma)$. 
If $y\not\in B_\sigma (x)$, then $\dist (y, U)>0$.
Set $r_1 (y):= \min\{r(y), \dist (y, U)\}$. Then
$\psi (1, \Sigma)=\Sigma$ on every annulus with center $y$
and radius smaller than $r_1 (y)$, and therefore it is $1/j$--a.m.
in it. If $y=x$, then the statement is obvious because
of Remark \ref{r:stillam}. If $y\in B_\sigma (x)\setminus \{x\}$,
then there exists $\rho (y), \tau (y)$ such that
$U\cup B_{\rho (y)} (y) \subset \an (x, \tau (y), \sigma)$.
By Remarks \ref{r:stillam} and \ref{r:contained},
$\psi (1, \Sigma)$ is $1/j$--a.m. on every annulus
centered at $y$ and outer radius smaller than $\rho (y)$.
\end{proof}

\begin{lemma}\label{l:uniquecont}
Let $\{\Sigma^j\}$ be a sequence as in Theorem
\ref{t:regularity} and $U$ and $\psi_j$ be as in Lemma \ref{l:stillam}.
Assume moreover that
$U$ is contained in a convex set $W$. If $\Sigma^j$
converges to a varifold $V$, then $\psi_j (1, \Sigma^j)$
converges as well to $V$.
\end{lemma}

\begin{proof}[Proof of Lemma \ref{l:uniquecont}]
By Theorem \ref{t:regularity} $V$ is a smooth minimal surface
(multiplicity allowed).
By Lemma \ref{l:stillam}, $\psi_j (1, \Sigma^j)$ is also
$1/j$--a.m. and again by Theorem \ref{t:regularity}
a subsequence (not relabeled) converges to a varifold
$V'$ which is a smooth minimal surface. Since
$\Sigma^j = \psi_j (1, \Sigma^j)$ outside $W$, $V=V'$ outside
$W$. Being $W$ convex, it cannot contain any closed 
minimal surface, and hence by standard unique continuation,
$V=V'$ in $W$ as well.
\end{proof}

\section{Proof of Proposition \ref{p:lifting}. Part II: Leaves}
\label{s:lifting}

\subsection{Step 1. Preliminaries}  Let $\{\Sigma^j\}$
be a sequence as in Theorem \ref{t:main2}. We
keep the convention that $\Gamma$ denotes the union of
disjoint closed connected embedded minimal
surfaces $\Gamma^i$ (with multiplicity $1$) and
that $\Sigma^j$ converges, in the sense of varifolds,
to $V= \sum_i n_i \Gamma^i$.
Finally, we fix a curve $\gamma$ contained in $\Gamma$.
 
Let $r: \Gamma \to \rn{+}$ be
such that the three conclusions of 
Proposition \ref{p:existbis} hold.
Consider a finite covering $\{B_{\rho_l} (x_l)\}$ of
$M$ with $\rho_l < r (x_l)$ and denote by $C$ 
the set of the centers $\{x_l\}$.
Next, up to extraction of
subsequences, we assume that the set of singular points
$P_j\subset \Sigma^j$ converges in the sense of Hausdorff
to a finite set $P$ (recall Remark \ref{r:P}) and we denote
by $E$ the union of $C$ and $P$. 
Recalling Remark \ref{r:contained},
for each $x\in M\setminus E$
there exists a ball $B$ centered at $x$
such that:
\begin{itemize}
\item $\Sigma^j\cap B$ is a smooth surface for
$j$ large enough;
\item $\Sigma^j$ is $1/j$--a.m. in $B$
for $j$ large enough.
\end{itemize}

Deform $\gamma$ to a smooth curve contained in
$\Gamma\setminus E$ and homotopic to $\gamma$ in $\Gamma$.
It suffices to prove the claim of the Proposition
for the new curve. By abuse of notation we continue to
denote it by $\gamma$. In what follows,
we let $\rho_0$ be any given positive number so small
that:
\begin{itemize}
\item $T_{\rho_0} (\Gamma)$ can be retracted on $\Gamma$;
\item For every $x\in \Gamma$,
$B_{\rho_0} (x)\cap \Gamma$ is a disk with diameter
smaller than the injectivity radius of $\Gamma$.
\end{itemize}
For any positive 
$\rho\leq 2\rho_0$ sufficiently small,
we can find a finite set of points
$x_1, \ldots, x_N$ on $\gamma$
with the following properties (to avoid cumbersome notation
we will use the convention $x_{N+1} = x_1$):
\begin{itemize}
\item[(C1)] If we let $[x_k, x_{k+1}]$ be the geodesic
segment on $\Gamma$ connecting $x_k$ and $x_{k+1}$, then
$\gamma$ is homotopic to $\sum_k [x_k, x_{k+1}]$.
\item[(C2)] $B_\rho (x_{k+1})\cap 
B_\rho (x_k)=\emptyset$;
\item[(C3)] $B_\rho (x_k)\cup B_\rho (x_{k+1})$
is contained in a ball $B^{k,k+1}$
of radius $3 \rho$;
\item[(C4)] In any ball $B^{k,k+1}$,
$\Sigma^j$ is $1/j$--a.m. and smooth provided $j$ is large enough;
\end{itemize}
see Figure \ref{f:snake}.
From now on we will consider $j$ so large that (C4)
holds for every $k$. 
The constant $\rho$ will be chosen (very small, but
independent of $j$)
only at the end of the proof. The existence of the
points $x_k$ is guaranteed by a simple compactness argument if
$\rho_0$ is a sufficiently small number.

\begin{figure}[htbp]
\begin{center}
    \input{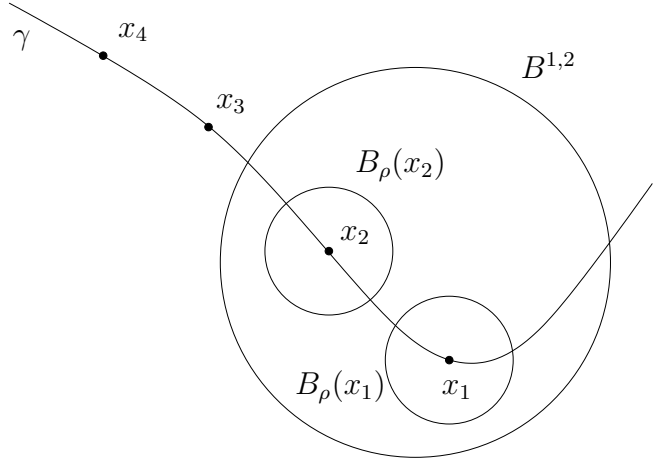}
    \caption{The points $x_l$ of (C1)-(C4).}
    \label{f:snake}
\end{center}
\end{figure}

\subsection{Step 2. Leaves} In every $B_\rho (x_k)$
consider a minimizing sequence $\Sigma^{j,l}:=
\psi_l (1, \Sigma^j)$ for 
Problem $(\Sigma^j, \Is_j (B_\rho (x_k),
\Sigma^j))$. Using Proposition \ref{p:MSYbis},
extract a subsequence converging 
(in $B_\rho (x_k)$) to
a smooth minimal surface $\Gamma^{j,k}$ with boundary
$\partial \Gamma^{j,k} = \Sigma^j\cap B_\rho (x_k)$.
This is a stable minimal surface, and we claim
that, as $j\uparrow \infty$, $\Gamma^{j,k}$ converges
smoothly on every ball $B_{(1-\theta)\rho} (x_k)$ (with
$\theta<1$) to $V$. Indeed, this is
a consequence of Schoen's curvature estimates, see Subsection
\ref{ss:curvest}. 

By a diagonal argument, if $\{l_j\}$ grows sufficiently
fast, $\Sigma^{j, l_j}\cap B_\rho (x_k)$ has the same limit 
as $\Gamma^{j,k}$.
On the other hand, for $\{l_j\}$ growing sufficiently fast,
Lemmas \ref{l:stillam} and \ref{l:uniquecont} apply,
giving that $\Sigma^{j,l_j}$ converges to $V$.

Therefore, $\Gamma^{j,k}$ converges smoothly to 
$n_i \Gamma^i \cap B_{(1-\theta)\rho} (x_k)$ 
in $B_{(1-\theta)\rho} (x_k)$ for
every positive $\theta<1$.
Therefore any connected
component of $\Gamma^{j,k}\cap B_{(1-\theta)\rho} (x_k)$ 
is eventually (for large $j$'s) a disk (multiplicity
allowed). 
The area of such a disk is,
by the monotonicity formula for minimal surfaces, at least
$c (1-\theta)^2\rho^2$, where $c$ is a constant depending
only on $M$. From now on we consider $\theta$ fixed, though
its choice will be specified later. 

Up to extraction of subsequences, we can assume that
for each connected component $\hat{\Sigma}^j$
of $\Sigma^j$, $\psi_l (1, \hat{\Sigma}^j)$ converges 
to a finite union of connected components
of $\Gamma^{j,k}$. However, in $B_{(1-\theta)\rho} (x_k)$,
\begin{itemize}
\item either their limit is zero;
\item or the area of $\psi_l (1, \hat{\Sigma}^j)$ in 
$B_{(1-\theta)\rho} 
(x_k)$ is larger than $c (1-2\theta)^2 \rho^2$
for $l$ large enough. 
\end{itemize}

We repeat this argument for every $k$. Therefore,
for any $j$ sufficiently large, we define
the set $\mathcal{L} (j,k)$ whose elements are those
connected components $\hat{\Sigma}^j$ of $\Sigma^j\cap
B_\rho (x_k)$ such that $\psi_l (1, \hat{\Sigma}^j)$
intersected with $B_{(1-\theta)\rho} (x_k)$ has area at
least $c (1-2\theta)^2 \rho^2$. 

Recall that $\Sigma^j$ is converging to 
$n_i\Gamma^i\cap B_\rho (x_k)$ in $B_\rho (x_k)$ in
the sense of varifolds. Therefore, the area of $\Sigma^j$
is very close to $n_i \haus^2 (\Gamma^i\cap B_\rho (x_k))$.
On the other hand, by definition $\haus^2 (\psi_l (1,
\Sigma^j)\cap B_\rho (x_k))$ is not larger.
This gives a bound to the cardinality of $\mathcal{L} (j,k)$,
independent of $j$ and $k$.
Moreover, if $\rho$ and $\theta$ are sufficiently small.
the constants $c$ and $\eps$ get so close, respectively, 
to $1$ and $0$ that the cardinality of
$\mathcal{L} (j, k)$ can be at most $n_i$. 

\subsection{Step 3. Continuation of the leaves}
We claim the following

\begin{lemma}[Continuation of the leaves]\label{l:cont}
If $\rho$ is sufficiently small, then for every
$j$ sufficiently large and for every element 
$\Lambda$ of $\mathcal{L} (j,k)$ there is an element
$\tilde{\Lambda}$ of $\mathcal{L} (j, k+1)$ such that
$\Lambda$ and $\tilde{\Lambda}$ are contained
in the same connected component of $\Sigma^j\cap
B^{k,k+1}$.
\end{lemma}

The lemma is sufficient to conclude the proof of
the Theorem. Indeed let $\{\Lambda_1, \Lambda_2, \ldots,
\Lambda_k\}$ be the elements of $\mathcal{L} (j,1)$.
Choose a point $y_1$ on $\Lambda_1$ and then a point
$y_2$ lying on an element $\tilde{\Lambda}$
of $\mathcal{L} (j,2)$ such that 
$\Lambda_1\cup \tilde{\Lambda}$ is contained in
a connected component of $\Sigma^j\cap B^{1,2}$. 
We proceed by induction and after $N$ steps
we get a point $y_{N+1}$ in some $\Lambda_k$. After
repeating at most $n_i+1$ times this procedure, we find
two points $y_{lN+1}$ and $y_{rN+1}$ belonging
to the same $\Lambda_s$. Without loss of generality
we discard the first $lN$ points and 
renumber the remaining ones 
so that we start with $y_1$ and end with
$y_{nN+1}= y_1$. Note that $n\leq n_i$. Each pair $y_k$, $y_{k+1}$
can be joined with a path $\gamma_{k,k+1}$
lying on $\Sigma^j$ and contained in a ball of radius $3\rho$,
and the same can be done with a path $\gamma_{nN+1, 1}$
joining $y_{nN+1}$ and $y_1$. Thus, if
we let 
$$
\tilde{\gamma} \;=\; \sum_k \gamma_{k,k+1} + \gamma_{nN+1, 1}
$$
we get a closed curve contained in $\Sigma^j$.

It is easy to show that the curve $\tilde{\gamma}$
is homotopic to $n\gamma$ in $\cup_k B^{k,k+1}$. 
Indeed, for each $sN+r$ fix a path $\eta^{sN+r}:[0,1]\to
B_\rho (x_r)$ with $\eta^{sN+r} (0) = y_{sN+r}$ and
$\eta^{sN+r} (1) = x_r$. Next fix an homotopy
$\zeta^{sN+r}: [0,1]\times [0,1]\to B^{k,k+1}$
with 
\begin{itemize}
\item $\zeta^{sN+r} (0, \cdot)= \gamma_{sN+r, sN+r+1}$,
\item $\zeta^{sN+r} (1, \cdot) = [x_r, x_{r+1}]$,
\item $\zeta^{sN+r} (\cdot, 0) = \eta^{iN+r} (\cdot)$
\item and $\zeta^{sN+r} (\cdot, 1) = \eta^{sN+r+1} (\cdot)$.
\end{itemize}
Joyning the $\zeta^k$'s we easily achieve an homotopy
between $\gamma$ and $\tilde{\gamma}$.
See Figure \ref{f:paths}.
If $\rho$ is chosen sufficiently small, then $\cup_k B^{k,k+1}$
is contained in a retractible tubular neighborhood
of $\Gamma$ and does not intersect $E$. 

\begin{figure}[htbp]
\begin{center}
    \input{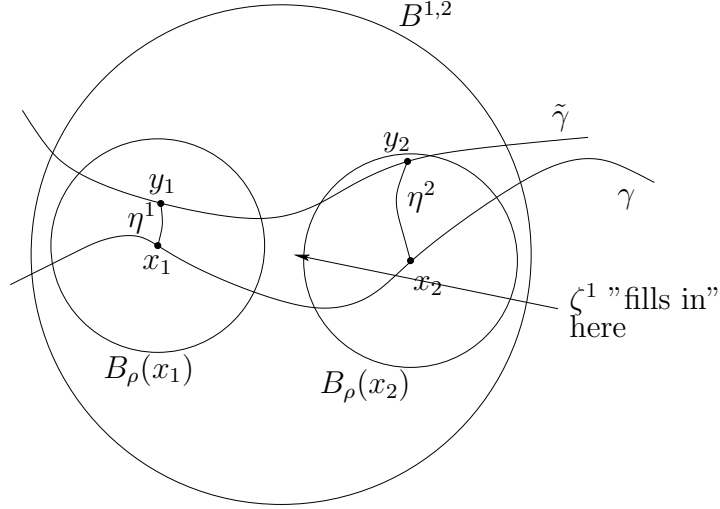}
    \caption{The homotopies $\zeta^{iN+r}$.}
    \label{f:paths}
\end{center}
\end{figure}

\subsection{Step 4. Proof of the Continuation of
the Leaves}

Let us fix a $\rho$ for
which Lemma \ref{l:cont} does not hold. Our goal
is to show that for $\rho$ sufficiently small,
this leads to a contradiction. Clearly,
there is an integer $k$ and a subsequence
$j_l\uparrow \infty$ such that the statement of
the Lemma fails. Without loss of generality we can assume
$k=1$ and we set $x=x_1$, $y=x_2$ and $B^{1,2} = B$. 
Moreover, by a slight abuse of notation
we keep labeling $\Sigma^{j_l}$ as $\Sigma^j$.

Consider the minimizing sequence of isotopies
$\{\psi_l\}$ for Problem $(\Sigma^j, \Is_j (B_\rho (x),
\Sigma^j))$ and $\{\phi_l\}$ for Problem $(\Sigma^j, \Is_j 
(B_\rho (y), \Sigma^j))$ fixed in Step 3.
Since $B_\rho (x)\cap B_\rho (y)
=\emptyset$ and $\psi_l$ and
$\phi_l$ leave, respectively, $M\setminus B_\rho (y)$
and $M\setminus B_\rho (x)$ fixed, we can combine
the two isotopies in
$$
\Phi_l (t,z)\;:=\;\left\{
\begin{array}{ll}
\psi_l (2t, z)  & \mbox{for $t\in [0,1/2]$}\\
\phi_l (2t-1,z) & \mbox{for $t\in [1/2,1]$.}
\end{array}\right.
$$
If we consider $\Sigma^{j,l} = \Phi_l (1, \Sigma^j)$, then
$\Sigma^{j,l}\cap B_\rho (x) = \psi_l (1, \Sigma^j)\cap B_\rho
(x)$ and $\Sigma^{j,l}\cap B_\rho (y) = \phi_l (1, \Sigma^j)
\cap B_\rho (y)$. Moreover for a sufficiently large 
$l$, the surface $\Sigma^{j,l}$ by Lemma \ref{l:stillam}
 is $1/j$--a.m. in $B$ and in sufficiently small annuli.

Arguing as in Step 2 (i.e. 
applying Theorem \ref{t:regularity}, Lemma \ref{l:stillam}
and Lemma \ref{l:uniquecont}), without loss of generality
we can assume that:
\begin{itemize}
\item[(i)] $\Sigma^{j,l}$ converges,
as $l\uparrow\infty$, to smooth minimal surfaces 
$\Delta^j$ and $\Lambda^j$ respectively
in $B_\rho (x)$ and $B_\rho (y)$;
\item[(ii)] $\Delta^j$ and $\Lambda^j$ converge,
respectively, to $n_i \Gamma^i\cap B_\rho (x)$
and $n_i \Gamma^i \cap B_\rho (y)$;
\item[(iii)] For $l_j$ growing sufficiently fast,
$\Sigma^{j,l_j}$ converges to the varifold 
$V = \sum_i n_i \Gamma^i$.
\end{itemize}
Let $\hat{\Sigma}^j$ be the connected component
of $\Sigma^j\cap B_\rho (x)$ which contradicts Lemma
\ref{l:cont}. Denote by $\tilde{\Sigma}^j$ the connected
component of $B\cap \Sigma^j$ containing $\hat{\Sigma}^j$.

Now, by Proposition \ref{p:MSYbis}, 
$\Phi_l (1, \tilde{\Sigma}^j)\cap B_\rho (x)$ 
converges to a stable minimal surface $\tilde{\Delta}^j
\subset \Delta^j$ and $\Phi_l (1, \hat{\Sigma}^j)$
converges to a stable minimal surface $\hat{\Delta}^j\subset
\tilde{\Delta}^j$. Because of (ii) and of curvature estimates
(see Subsection \ref{ss:curvest}),
$\hat{\Delta}^j$ converges necessarily
to $r \Gamma^i \cap B_\rho (x)$ for some integer $r\geq 0$.
Since $\hat{\Sigma}^j\in \mathcal{L} (j,1)$,
it follows that $r\geq 1$. 
Similarly, $\Phi_l (1, \tilde{\Sigma}^j)\cap B_\rho (y)$
converges to a smooth minimal surface $\tilde{\Lambda}^j$
and $\tilde{\Lambda}^j$ converges to $s\Gamma^i\cap
B_\rho (y)$ for some integer $s\geq 0$. Since
$\tilde{\Sigma}^j$ does not contain any element
of $\mathcal{L} (j,2)$, it follows necessarily $s=0$.

Consider now the varifold $W$ which is the limit
in $B$ of  
$\tilde{\Sigma}^{j,l_j} = \Phi_{l_j} (1, \tilde{\Sigma}^j)$.
Arguing again as in Step 2 we choose $\{l_j\}$
growing so fast that $W$, which is the limit of 
$\tilde{\Sigma}^{j, l_j}$,
coincides with the limit of $\tilde{\Delta}^j$ in $B_\rho (x)$
and with the limit of $\tilde{\Lambda}_j$ in $B_\rho (y)$.
According to the discussion above, 
$V$ coincides then with $r\Gamma^i \cap B_\rho (x)$
in $B_\rho (x)$ and vanishes in $B_\rho (y)$. Moreover
\begin{equation}\label{e:ineq}
\|W\|\;\leq\; \|V\|\res B \;=\; n \haus^2\res \Gamma^i\cap B
\end{equation} 
in the sense of varifolds. 
We recall here that $\|W\|$ and $\|V\|\res B$ are nonnegative
measures defined in the following way:
\begin{equation}\label{e:varifold}
\int \varphi (x) d\|W\| (x)
\;=\; \lim_{j\uparrow\infty}
\int_{\tilde{\Sigma}^{j,l_j}} \varphi
\end{equation}
and
\begin{equation}\label{e:varifold1}
\int \varphi (x) d\|V\| (x)
\;=\; \lim_{j\uparrow\infty}
\int_{\Sigma^{j,l_j}} \varphi
\end{equation}
for every $\varphi\in C_c (B)$. Therefore
\eqref{e:ineq} must be understood as a standard inequality
between measures, which is an effect of \eqref{e:varifold},
\eqref{e:varifold1} and the inclusion $\tilde{\Sigma}^{j,l_j}
\subset \Sigma^{j,l_j}\cap B$. An important consequence
of \eqref{e:ineq} is that
\begin{equation}\label{e:zero}
\|W\| (\partial B_\tau (w))\;=\; 0
\qquad\mbox{for every ball $B_\tau (w)\subset B$.}
\end{equation}

Next, consider the geodesic segment $[x,y]$ joining
$x$ and $y$ in $\Gamma^i$. For $z\in [x,y]$,
$B_{\rho/2} (z)\subset B$. Moreover, 
\begin{equation}\label{e:continuity}
\mbox{the map}\quad
z\;\;\mapsto\;\; \|W\| (B_{\rho/2} (z))\quad  
\mbox{is continuous in $z$,}
\end{equation}
because of \eqref{e:ineq} and \eqref{e:zero}.

Since $\|W\| (B_{\rho/2} (x))\geq \haus^2 (\Gamma^i
\cap B_{\rho/2} (x))$
and $\|W\| (B_{\rho/2} (y))=0$, by
the continuity of the map in \eqref{e:continuity},
there exists $z\in [x,y]$ such that 
$$
\|W\| (B_{\rho/2} (z))\;=\;\frac{1}{2}\haus^2 (\Gamma^i
\cap B_{\rho/2} (z))\, .
$$
Since $\|W\| (\partial B_{\rho/2}(z)) =0$, we conclude
(see Proposition 1.62(b) of \cite{AFP}) that
\begin{equation}\label{e:est1}
\lim_{j\uparrow \infty}
\haus^2 (\tilde{\Sigma}^{j,l_j}\cap B_{\rho/2} (z))
\;=\; \frac{1}{2} \haus^2 (\Gamma^i
\cap B_{\rho/2} (z))\, 
\end{equation}
(see Figure \ref{f:W}).

\begin{figure}[htbp]
\begin{center}
    \input{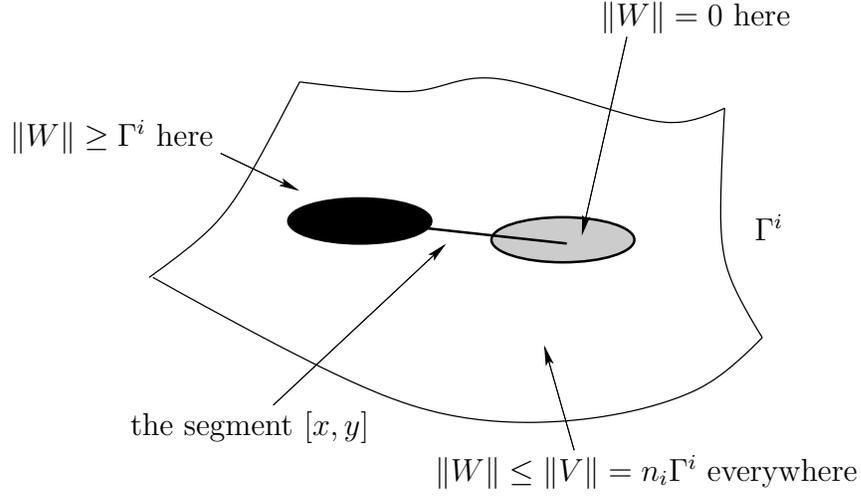}
    \caption{The varifold $W$.}
    \label{f:W}
\end{center}
\end{figure}

On the other hand, since $\Sigma^{j, l_j}$ converges
to $V$ in the sense of varifolds and $V= n_i 
\Gamma^i\cap B_{\rho/2} (z)$ in $B_{\rho/2} (z)$,
we conclude that
\begin{equation}\label{e:est2}
\lim_{j\uparrow \infty}
\haus^2 ((\Sigma^{j,l_j}\setminus 
\tilde{\Sigma}^{j,l_j})\cap B_{\rho/2} (z))
\;=\; \left(n_i-\frac{1}{2}\right) \haus^2 (\Gamma^i
\cap B_{\rho/2} (z))\, .
\end{equation}
If $\rho$ is sufficiently small, $\Gamma^i\cap
B_{\rho/2} (z)$ is close to a flat disk and
$B_{\rho/2} (z)$ is close to a flat ball. 

Using the coarea formula and Sard's lemma, we can
find a $\sigma\in ]0, \rho/2[$ and a subsequence
of $\{\Sigma^{j, l_j}\}$ (not relabeled) with the following
properties:
\begin{itemize}
\item[(a)] $\Sigma^{j, l_j}$ intersects $\partial B_\sigma (z)$
transversally;
\item[(b)] $\Length (\tilde{\Sigma}^{j, l_j}\cap \partial 
B_\sigma (z)) \leq 2 (1/2+\eps) \pi \sigma$;
\item[(c)] $\Length ((\Sigma^{j, l_j}\setminus
\tilde{\Sigma}^{j, l_j})\cap \partial 
B_\sigma (z)) \leq 2 ((n_i-1/2)+\eps)\pi \sigma$;
\item[(d)] $\haus^2 (\Gamma^i \cap B_\sigma (z))\geq (1-\eps)
\pi \sigma^2$.
\end{itemize}
Note that the geometric constant $\eps$ can be made as close
to $0$ as we want by choosing $\rho$ sufficiently small.

In order to simplify the notation, set $\Omega^j =
\Sigma^{j,l_j}$. Consider a minimizing sequence
$\Omega^{j,s} = \varphi_s (1, \Omega^j)$ 
for Problem $(\Omega^j,
\Is_j (B_\sigma (z), \Omega^j))$. By Proposition
\ref{p:MSYbis}, $\Omega^{j,s}\cap B_\sigma (z)$
converges, up to subsequences, 
to a minimal surface $\Xi^j$ with boundary
$\Omega^j\cap \partial B_\sigma (z)$. Moreover, using
Lemma \ref{l:uniquecont} and arguing as in 
the previous steps, we conclude that $\Xi^j$
converges to $n_i\Gamma^i\cap B_\sigma (z)$. 

Next, set:
\begin{itemize}
\item $\tilde{\Omega}^j = 
\tilde{\Sigma}^{j, l_j}\cap B_\sigma (z)$,
$\tilde{\Omega}^{j,s}
= \varphi_s (1, \tilde{\Omega}^j)$;
\item $\hat{\Omega}^j =
(\Sigma^{j,l_j}\setminus 
\tilde{\Sigma}^{j, l_j})\cap B_\sigma (z)$,
$\hat{\Omega}^{j,s}
= \varphi_s (1, \hat{\Omega}^j)$.
\end{itemize}
By Proposition \ref{p:MSYbis}, since $\tilde{\Omega}^j$
and $\hat{\Omega}^j$ are unions of connected components
of $\Omega^j\cap B_\sigma (z)$, we can assume that 
$\tilde{\Omega}^{j,s}$ and 
$\hat{\Omega}^{j,s}$ converge respectively
to stable minimal surfaces
$\tilde{\Xi}^j$ and $\hat{\Xi}^j$
with 
$$
\partial \tilde{\Xi}^j
\;=\; \tilde{\Sigma}^{j,l_j}\cap \partial B_\sigma (z)
\qquad
 \partial \hat{\Xi}^j
\;=\; (\Sigma^{j,l_j}\setminus \tilde{\Sigma}^{j,l_j})
\cap \partial B_\sigma (z)\, .
$$
Hence, by (b) and (c), we have
\begin{equation}\label{e:est10}
\Length (\partial \tilde{\Xi}^j)
\;\leq\; 2 \left(\frac{1}{2}+\eps\right) \pi
\sigma\qquad
\Length (\partial \hat{\Xi}^j)
\;\leq\; 2 \left(n_i-\frac{1}{2}+\eps\right)
\pi \sigma\, .
\end{equation}
On the other hand, using the standard monotonicity
estimate of Lemma \ref{l:monot} below, we conclude that
\begin{equation}\label{e:est11}
\haus^2 (\hat{\Xi}^j) \;\leq\;
\left(n_i - \frac{1}{2} +\eta\right)\pi \sigma^2
\end{equation}
\begin{equation}\label{e:est12}
\haus^2 (\tilde{\Xi}^j) \;\leq\;
\left(\frac{1}{2} +\eta\right)\pi \sigma^2\, .
\end{equation}
As the constant $\eps$ in (d), $\eta$ as well
can be made arbitrarily small by choosing $\rho$
suitably small. We therefore choose $\rho$ so small that
\begin{equation}\label{e:est20}
\haus^2 (\hat{\Xi}^j) \;\leq\;
\left(n_i - \frac{3}{8}\right)\pi \sigma^2\, ,
\end{equation}
\begin{equation}\label{e:est21}
\haus^2 (\tilde{\Xi}^j) \;\leq\;
\frac{5}{8}\pi \sigma^2\, 
\end{equation}
and
\begin{equation}\label{e:est22}
\haus^2 (\Gamma^i\cap B_\sigma (z)) \;\geq\;
\left(1-\frac{1}{8n_i}\right)\pi \sigma^2\, .
\end{equation}

Now, by curvature estimates (see Subsection
\ref{ss:curvest}), we can assume that the
stable minimal surfaces $\tilde{\Xi}^j$ and $\hat{\Xi}^j$,
are converging smoothly (on compact subsets of 
$B_\sigma (z))$ to stable minimal surfaces
$\tilde{\Xi}$ and $\hat{\Xi}$. Since $\Xi^j=
\tilde{\Xi}^j+\hat{\Xi}^j$
converges to $n_i \Gamma^i \cap B_\sigma (z)$,
we conclude that
$\tilde{\Xi} = \tilde{n} \Gamma^i \cap B_\sigma (z)$
and $\hat{\Xi} = \hat{n} \Gamma^i\cap B_\sigma (z)$,
where $\tilde{n}$ and $\hat{n}$ are nonnegative integers
with $\tilde{n}+\hat{n}=n_i$. On the other hand, by
\eqref{e:est20}, \eqref{e:est21} and \eqref{e:est22},
we conclude
\begin{equation}\label{e:est30}
\tilde{n} \left(1-\frac{1}{8n_i}\right)\pi \sigma^2
\;=\; \haus^2 (\tilde{\Xi})
\;\leq\; \liminf_j \haus^2 (\tilde{\Xi}^j)
\;\leq\; \frac{5}{8} \pi \sigma^2
\end{equation}
\begin{equation}\label{e:est31}
\hat{n} \left(1-\frac{1}{8n_i}\right) \pi\sigma^2
\;=\; \haus^2 (\hat{\Xi})
\;\leq\; \liminf_j \haus^2 (\hat{\Xi}^j)
\;\leq\; \left(n_i - \frac{3}{8}\right)\pi\sigma^2\, .
\end{equation}
From \eqref{e:est30} and \eqref{e:est31} we conclude, respectively,
$\tilde{n}=0$ and $\hat{n}\leq n_i-1$,
which contradicts $\tilde{n}+\hat{n} = n_i$.

\subsection{A simple estimate} The following
lemma is a standard fact in the theory of minimal surfaces.

\begin{lemma}\label{l:monot}
There exist constants $C$ and $r_0>0$ (depending only on $M$)
such that
\begin{equation}\label{e:monot}
\haus^2 (\Sigma)
\;\leq\; \left(\frac{1}{2} +C\sigma\right) \sigma \Length\, 
(\partial \Sigma)\, 
\end{equation}
for any $\sigma<r_0$ and for 
any smooth minimal surface $\Sigma$ with boundary $\partial \Sigma
\subset \partial B_\sigma (z)$.
\end{lemma} 

Indeed, \eqref{e:monot} follows
from the usual computations leading to the monotonicity formula.
However, since we have not found a reference
for \eqref{e:monot} in the literature, we will sketch
a proof in Appendix \ref{a:monot}.

\section{Proof of Proposition \ref{p:MSYbis}. Part I:
Convex hull property}\label{s:MSY1}

\subsection{Preliminary definitions} Consider an open geodesic ball
$U= B_\rho (\xi)$ with sufficiently small radius $\rho$ and a subset $\gamma\subset
\partial U$ consisting of finitely many disjoint smooth Jordan curves.  



\begin{definition}\label{d:transversal}
We say that an open subset $A\subset U$
meets $\partial U$ in $\gamma$ transversally if there exists a
positive angle $\theta_0$ such that:
\begin{itemize}
\item[(a)] $\partial A\cap \partial U
\subset \gamma$.
\item[(b)] For every $p\in \partial A\cap \partial U$ we choose coordinates $(x,y,z)$
in such a way that the tangent plane $T_p$ of $\partial U$ at $p$ is 
the $xy$-plane and $\gamma'(p)=(1,0,0)$. Then in this setting every point $q=(q_1,q_2,q_3)\in A$ satisfies 
$\frac{q_3}{q_2}\geq \tan (\textstyle{\frac{1}{2}}- \theta_0).$
\end{itemize}
\end{definition}

\begin{remark}
Condition (b) of the above definition can be stated in the following 
geometric way: 
there exixt two halfplanes $\pi_1$ and $\pi_2$ meeting at the line 
through $p$ in direction $\gamma'(p)$ such that 
\begin{itemize}
\item they form an angle $\theta_0$ with $T_p$; 
\item the set $A$ is all contained in the wedge formed by 
$\pi_1$ and $\pi_2$;
\end{itemize}
see Figure \ref{f:wedge}.
 \end{remark}


\begin{figure}[htbp]
\begin{center}
    \input{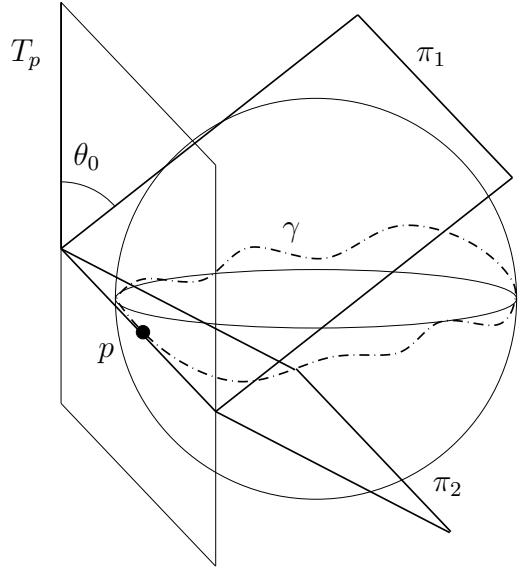}
    \caption{For any $p\in A \cap
\partial U$, $A$ is contained in a wedge delimited by two
halfplanes meeting at $p$ transversally to the plane $T_p$.}  
    \label{f:wedge}
\end{center}
\end{figure}

In this section we will show the following lemma.

\begin{lemma}[Convex hull property]\label{l:convexhull} Let $V$
and $\Sigma$ be as in Proposition \ref{p:MSYbis}. 
Then, there exists a convex open set
$A\subset U$ which intersects $U$ in $\partial \Sigma$ transversally
and such that $\supp (\|V\|)\subset \overline{A}$.
\end{lemma}

Our starting point is the following elementary fact about
convex hulls of smooth curves lying 
in the euclidean two--sphere. 

\begin{propos}\label{p:C^2} If $\beta \subset \partial \mathcal{B}_1
\subset \rn{3}$ is the union of finitely many $C^2-$Jordan curves, then its 
convex hull meets $\mathcal{B}_1$ 
transversally in $\beta$.
\end{propos}

The proof of this proposition follows from the regularity and the compactness 
of $\beta$ and from the fact that $\beta$ is not self-intersecting. 
We leave its details to the reader.






\subsection{Proof of Lemma \ref{l:convexhull}}
From now on, we consider $\gamma = \partial \Sigma$:
this is the union of finitely many disjoint smooth
Jordan curves contained in $\partial U$. Recall
that $U$ is a geodesic ball $B_\rho (\xi)$. 
Without loss of generality we assume that $\rho$ is
smaller than the injectivity radius.

\medskip

{\bf Step 1} Consider the rescaled exponential coordinates
induced by the chart $f:\overline{B}_\rho (\xi)\to
\overline{\B}_1$ given by $f (z)=(\exp_\xi^{-1} (z))/\rho$. These
coordinates will be denoted by $(x_1, x_2, x_3)$.
We apply Proposition \ref{p:C^2} and consider the
convex hull $B$ of $\beta = f (\partial \Sigma)$ in $\mathcal{B}_1$. According
to our definition, $f^{-1} (B)$ meets $U$ transversally 
in $\gamma$. 

We now let $\theta_0$ be a positive angle
such that condition (b) in Definition \ref{d:transversal}
is fulfilled for $B$. 
Next we fix a point $x\in f (\gamma)$ and consider
consider the halfplanes $\pi_1$ and $\pi_2$ delimiting
the wedge of condition (b). Without loss of generality,
we can assume that the coordinates are
chosen so that $\pi_1$ is given by 
$$
\pi_1 = \{(z_1, z_2, z
_3): z_3\leq
a\}
$$ 
for some positive constant $a$. Condition (b) ensures
that $a\leq a_0<1$ for some constant $a_0$ inpendent of
the point $x\in f (\gamma)$.

For $t\in ]0, \infty[$ denote by $C_t$ the points $C_t := \{(0,0,
-t)\}$ and by $r(t)$ the positive real numbers
$$
r(t) \;:=\;\sqrt{1+t^2+2at}
$$
We finally denote by $R_t$ the closed balls
$$
R_t\;:=\; \overline{\mathcal{B}}_{r(t)} (C_t)\, .
$$
The centers $C_t$ and the radii $r (t)$ are chosen in such
a way that the intersection of the sphere
$\partial R_t$ and $\partial \mathcal{B}_1$
is always the circle $\pi_1 \cap \partial \mathcal{B}_1$.

\begin{figure}[htbp]
\begin{center}
    \input{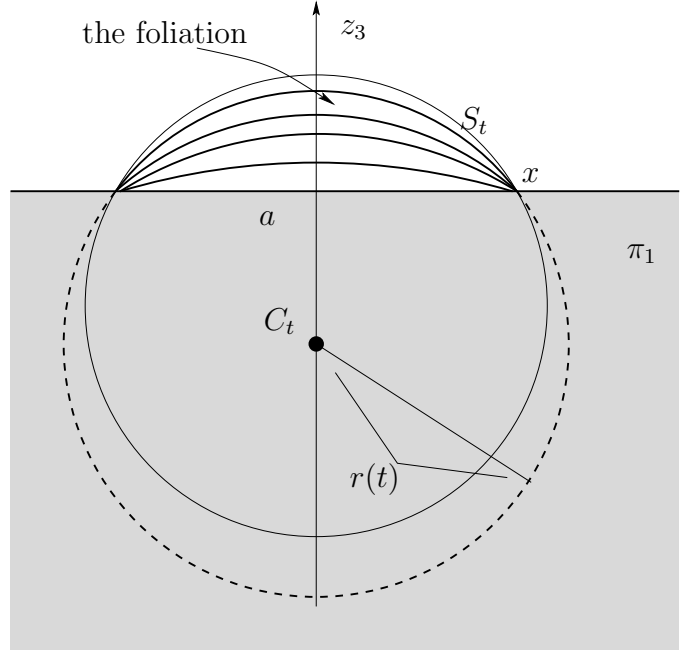}
    \caption{A planar cross-section of the foliation 
$\{S_t : t\in ]0,\infty[\}$.}
    \label{f:foliation}
\end{center}
\end{figure}

Note, moreover, that for $t$ coverging to $+\infty$, the
ball $R_t$ converges towards the region
$\{z_3\leq a\}$.
Therefore, the region $\{z_3>a\} \cap \mathcal{B}_1$ is foliated  with the caps
$$
S_t \;:=\; \partial R_t \cap \mathcal{B}_1\, \qquad \mbox{for
$t\in ]0, \infty[$}.
$$
In Figure \ref{f:foliation}, we see a section of this foliation
with the plane $z_2z_3$. 

\medskip

We claim that, for some constant $t_0>0$ independent
of the choice of the point $x\in f(\gamma)$,
the varifold $V$ is supported in $f^{-1} (R_{t_0})$. 
A symmetric procedure can be followed starting from
the plane $\pi_2$. In this way we find two 
off-centered balls and hence a corresponding wedge
$W_x$ satisfying condition (b) of Definition \ref{d:transversal}
and containing the support of $V$; see Picture \ref{f:13}. 
Our claim that the constant $t_0$ can be chosen independently 
of $x$ and the bound $a\leq a_0 <1$ imply that the
the planes delimiting the wedge $W_x$ form an angle larger 
than some fixed constant with
the plane $T_x$ tangent to $\partial \mathcal{B}_1$ at $x$.
Therefore, the intersections of all the wedges $W_x$, for $x$ varying
among the points of $\gamma$, yield the desired set $A$. 

\begin{figure}[htbp]
\begin{center}
    \input{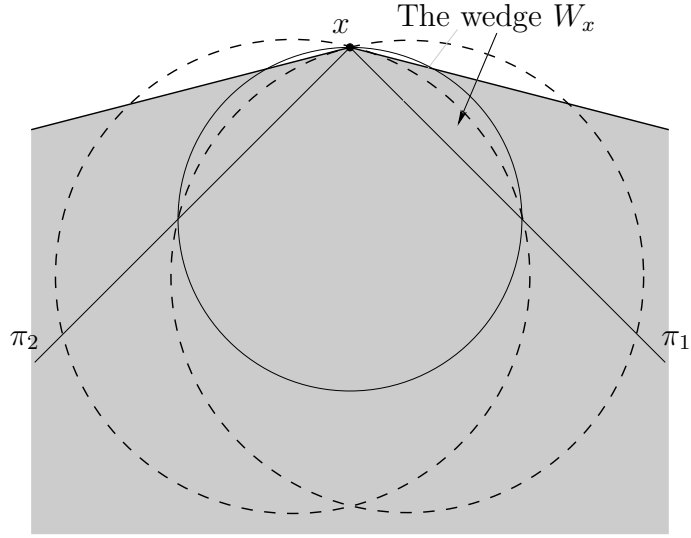}
    \caption{A planar cross-section of the wedge $W_x$.}
    \label{f:13}
\end{center}
\end{figure}

\medskip

{\bf Step 2} We next want to show that the varifold $V$ is supported
in the closed ball $f^{-1} (R_{t_0})$. For any $t\in [0, t_0[$,
denote by $\pi_t: \overline{U} \to f^{-1} (R_{t})$ 
the nearest point projection.
If the radius $\rho_0$ of $U$ and the parameter $t_0$ 
are both sufficiently small, then $\pi_t$ is a well defined 
Lipschitz map 
(because there exists a unique nearest point). Moreover,
the Lipschitz constant of $\pi_t$ is equal to $1$ and,
for $t>0$, $|\nabla \pi_t|<1$ on $U\setminus f^{-1} (R_t)$.
In fact the following lemma holds.

\begin{lemma}\label{l:technical_projection}
Consider in the euclidean ball $\mathcal{B}_1$ a set $U$
that is uniformly convex, with constant $c_0$. Then there
is a $\rho (c_0)>0$ such that, if 
$\rho_0\leq \rho (c_0)$, then the nearest point projection
$\pi$ on $\overline{f (U)}$ is a Lipschitz map with constant
$1$. Moreover, at every point $P\not\in \overline{f(U)}$,
$|\nabla \pi (P)|<1$.
\end{lemma}

The proof is elementary and we give it in Appendix \ref{a:proj}
for the reader's convenience. Next, it is obvious that
$\pi_0$ is the identity map and that the map
$(t,x)\mapsto \pi_t (x)$ is smooth.

Assume now for a contradiction that $V$ is not supported in 
$f^{-1} (R_{t_0})$.
By Lemma \ref{l:technical_projection}, 
the varifold $(\pi_{t_0})_{\#} V$ has, therefore, strictly less
mass than the varifold $V$. 

Next, consider a minimizing sequence $\Delta^k$ as 
in the statement
of proposition \ref{p:MSYbis}. Since $\partial \Delta^k = \partial \Sigma$,
the intersection of $\overline{\Delta^k}$ with $\partial U$ 
is given by $\partial \Sigma$. On the other hand,
by construction $\partial \Sigma \subset f^{-1} (R_t)$ and
therefore, if we consider $\Delta^k_t := (\pi_t)_\# \Delta^k$ we obtain
a (continuous) one-parameter family of currents with the properties that
\begin{itemize}
\item[(i)] $\partial \Delta^k_t = \partial \Sigma$;
\item[(ii)] $\Delta^k_0 = \Delta_0$;
\item[(iii)] The mass of $\Delta^k_t$ is less or equal than $\haus^2 (\Delta^k)$;
\item[(iv)] The mass of $\Delta^k_{t_0}$ converges towards the mass
of $(\pi_{t_0})_\# V$ and hence, for $k$ large
enough, it is strictly smaller than the
mass of $V$.
\end{itemize}

Therefore, if we fix a sufficiently large number $k$, we can assume 
that (iv) holds with a gain in mass of a positive amount $\eps=1/j$.
We can, moreover, assume that $\haus^2 (\Delta^k) \leq \haus^2 (\Sigma)
+ 1/(8j)$. 
By an approximation procedure, it is possible to replace
the family of projections $\{\pi_t\}_{t\in [0, t_0]}$ with a smooth
isotopy $\{\psi_t\}_{t\in [0,1]}$ with the following properties:
\begin{itemize}
\item[(v)] $\psi_0$ is the identity map and $\psi_t|_{\partial U}$
is the identity map for every $t\in [0, 1]$;
\item[(vi)] $\haus^2 (\Delta^k) \leq \haus^2 (\psi_t 
(\Sigma)) +1/(8j)$;
\item[(vii)] $\haus^2 (\psi_1 (\Delta^k))
\leq {\bf M} ((\pi_{t_0})_\# V) -1/j$.
\end{itemize}
This contradicts the $1/j$--almost minimizing property of $\Sigma$.

\medskip

In showing the existence of the family of isotopies $\psi_t$,
a detail must be taken into account: the map $\pi_t$
is smooth everywhere on $\overline{U}$ but on the circle 
$f^{-1} (R_t)\cap \partial U$ (which is the same circle for every $t$!).
We briefly indicate here a procedure to construct $\psi_t$, skipping
the cumbersome details. 
 
We replace the sets $\{R_t\}$ with a new family $\mathcal{R}_t$ which have
the following properties:
\begin{itemize}
\item $\mathcal{R}_0 = \overline{\mathcal{B}}_1$;
\item $\mathcal{R}_{t_0} = R_{t_0}$;
\item For $t\in [0, t_0]$ the boundaries $\partial \mathcal{R}_t$
are uniformly convex;
\item $\partial \mathcal{R}_t \cap \partial \mathcal{B}_1
= R_t \cap \partial \mathcal{B}_1$;
\item The boundaries of $\partial \mathcal{R}_t$ are smooth 
for $t\in [0, t_0[$ and form a smooth foliation of $\mathcal{B}_1 (0)
\setminus R_{t_0}$.
\end{itemize}
The properties of the new sets are illustrated in Figure 
\ref{f:tocca}

\begin{figure}[htbp]
\begin{center}
    \input{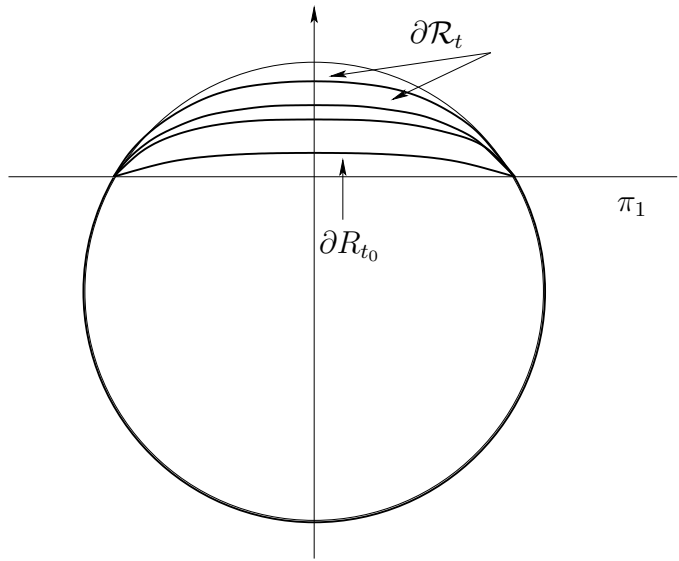}
    \caption{A planar cross-section of the new foliation.}
    \label{f:tocca}
\end{center}
\end{figure}

Since $\overline{\Delta^k}$ touches $\partial U$ in $\partial \Sigma$
transversally
and $\partial \Sigma\subset f^{-1} (\mathcal{R}_t)$ for every $t$,
we conclude the existence of a small $\delta$ such that 
$\Delta^k \subset f^{-1} (\mathcal{R}_{2\delta})$. Moreover, for
$\delta$ sufficiently small, the nearest point projection
$\tilde{\pi}_{t_0-\delta}$ on 
$f^{-1} (\mathcal{R}_{t_0-\delta})$ is so close to $\pi_{t_0}$
that 
$$
{\bf M} ((\tilde{\pi}_{t_0-\delta})_\# \Delta^k)
\;\leq\; {\bf M} ((\pi_{t_0})_\# \Delta^k) + \eps/4\, .
$$

We then construct $\psi_t$ in the following way. We fix
a smooth increasing bijective function $\tau: [0,1]\to [\delta, t_0-\delta]$,
\begin{itemize}
\item $\psi_t$ is the identity on $\overline{U}\setminus \mathcal{R}_\delta$
and on $\mathcal{R}_{\tau (t)}$;
\item On $\mathcal{R}_\delta\setminus \mathcal{R}_{\tau (t)}$
it is very close to the projection $\tilde{\pi}_{\tau (t)}$
on $\mathcal{R}_{\tau (t)}$.
\end{itemize}
In particular, for this last step, we fix for a smoooth function
$\sigma: [0,1]\times [0,1]$ such that, for each $t$,
$\sigma (t, \cdot)$ is a smooth bijection between $[0,1]$
and $[\delta, \tau (t)]$ very close to the function
which is identically $\tau (t)$ on $[0, 1]$.
Then, for $s\in [0,1]$, we define $\psi_t$ on the surface 
$\partial\mathcal{R}_{(1-s)\delta + s \tau (t)}$ to be
the nearest point projection on the surface
$\partial\mathcal{R}_{\sigma (t, s)}$. So, $\psi_t$
fixes the leave $\partial \mathcal{R}_\delta$ but
moves most of the leaves between $\partial \mathcal{R}_\delta$
and $\partial \mathcal{R}_{\tau (t)}$ towards 
$\partial \mathcal{R}_{\tau (t)}$. This completes the proof of Lemma \ref{l:convexhull}.

\section{Proof of Proposition \ref{p:MSYbis}. Part II:
Squeezing Lemma}\label{s:MSY2} In this section we prove the
following Lemma.

\begin{lemma}[Squeezing Lemma]\label{l:squeeze}
Let $\{\Delta^k\}$ be as in Proposition \ref{p:MSYbis}, $x\in
\overline{U}$ and $\beta>0$ be given. Then there exists an
$\eps_0>0$ and a $K\in \N$ with the following property. If $k\geq K$
and $\varphi\in \Is (B_{\eps_0} (x)\cap U)$ is such that $\haus^2
(\varphi (1, \Delta^k))\leq \haus^2 (\Delta^k)$, then 
there exists a $\Phi\in \Is (B_{\eps_0} (x)\cap U)$ such that
\begin{equation}\label{e:finale}
\Phi (1, \cdot)\;=\; \varphi (1, \cdot)
\end{equation}
\begin{equation}\label{e:durante}
\haus^2 (\Phi (t, \Delta^k))\;\leq\; \haus^2 (\Delta^k) +
\beta\qquad \mbox{for every $t\in [0,1]$.}
\end{equation}
\end{lemma}

If $x$ is an interior point of $U$, this lemma reduces to Lemma 7.6
of \cite{CD}. When $x$ is on the boundary of $U$, one can argue in a
similar way (cp. with Section 7.4 of \cite{CD}). Indeed, the proof
of Lemma 7.6 of \cite{CD} relies on the fact that, when $\eps$ is
sufficiently small, the varifold $V$ is close to a cone. For
interior points, this follows from the stationarity of the varifold
$V$. For points at the boundary this,
thanks to a result of Allard (see \cite{All2}), is a consequence 
of the stationarity of $V$ and of the convex hull property of
Lemma \ref{l:convexhull}.

\subsection{Tangent cones}\label{ss:tcones} 
Consider the varifold $V$ of
Proposition \ref{p:MSYbis}. Given a point $x\in \overline{U}$
and a radius $\rho>0$, consider the chart 
$f_{x,\rho} : B_\rho (x)\to \mathcal{B}_1$ given by
$f_{x,\rho} (y) = \exp_x^{-1} (y)/\rho$. 
We then consider the varifolds $V_{x,\rho} := (f_{x,\rho})_\# V$.
Moreover, if $\lambda>0$, we will denote by $O_\lambda:
\rn{3}\to \rn{3}$ the rescaling $O_\lambda (x) = x/\lambda$.

If $x\in U$, the monotonicity formula and a
compactness result (see Theorem 19.3 of \cite{Si}) imply that, 
for any $\rho_j\downarrow 0$, there
exists a subsequence, not relabeled, such that
$V_{x, \rho_j}$ converges to an integer rectifiable 
varifold $W$ supported in $\mathcal{B}_1$ with the property
that $(O_\lambda)_\# W \res B_1 (0) = W$ for any $\lambda<1$. 
The varifolds $W$ which are limit of subsequences $V_{x, \rho_j}$
are called tangent cones to $V$ at $x$. The monotonicity formula
implies that the mass of each $W$ is a positive
constant $\theta (x, V)$ independent of $W$ (see again Theorem 19.3 of \cite{Si}).

If $x\in \partial U$, we fix coordinates $y_1, y_2, y_3$
in $\R^3$ in such
a way that $f_{x, \rho} (U\cap B_\rho (x))$ converges to
the half-ball $\mathcal{B}_1^+ = \mathcal{B}_1 \cap \{y_1>0\}$.

Recalling Lemma \ref{l:convexhull}, 
we can infer with the monotonicity formula of Allard for
points at the boundary (see 3.4 of \cite{All2}) 
that $V_{x, \rho} = (f_{x, \rho})_\# V$ 
have equibounded mass. Therefore, if $\rho_j\downarrow 0$,
a subsequence of $V_{x, \rho_j}$, not relabeled, converges
to a varifold $W$.

By Lemma \ref{l:convexhull}, there
is a positive angle $\theta_0$ such that, after
a suitable change of coordinates, $W$ is supported
in the set
$$
\{|y_2| \leq y_1 \tan \theta_0\}\, .
$$
Therefore $\supp (W)\cap \{y_1=0\}= \{(0,0,t): t\in [-1,1]\}
=: \ell$. Applying the monotonocity formula of 3.4 of
\cite{All2}, we conclude that
\begin{equation}\label{e:noncarica}
\|W\| (\ell) \;=\; 0\, 
\end{equation}
and
\begin{equation}\label{e:massariscala}
\|W\| (\mathcal{B}_\rho (0)) \;=\; \pi \theta (\|V\|, x) \rho^2\, ,
\end{equation}
where
$$
\theta (\|V\|, x) \;=\; \lim_{r\downarrow 0} \frac{\|V\| (B_\rho (x))}{\pi \rho^2}
$$
is independent of $W$. Being $W$ the limit of a sequence $V_{x, \rho_j}$ with
$\rho_j\downarrow 0$, we conclude that $W$ is a stationary varifold. 

Now, define the reflection map $r: \rn{3}\to \rn{3}$ given by
$r (z_1, z_2, z_3) = (-z_1, -z_2, z_3)$. 
By
\eqref{e:noncarica}, using the reflection principle of 3.2 of \cite{All2},
the varifold $W':= W+ r_\# W$ is a stationary varifold. By 
\eqref{e:massariscala} and Corollary 2 of 5.1 in \cite{All}, we conclude
that $(O_\lambda)_\# W' \res \mathcal{B}^+_1 = W'$
for every $\lambda<1$. On the other hand, this implies $(O_\lambda)_\# W \res 
\mathcal{B}^+_1 = W$. Therefore $W$ is a cone and we will call
it {\em tangent cone to $V$ at $x$}.

\subsection{A squeezing homotopy}\label{ss:squeeze}
Since for points in the interior
the proof is already given in \cite{CD}, we assume that $x\in
\partial U$. Moreover, the proof given here in this case can easily
be modified for $x\in U$. 
Therefore we next fix a small radius
$\eps> 0$ and consider an isotopy
$\varphi$ of $U\cap B_\eps (x)$ keeping the boundary fixed. 

We start by fixing a small parameter $\delta>0$ which will be chosen
at the end of the proof. Next, we consider a
diffeomorphism $G_\eps$ between $\mathcal{B}_\eps^+
= \mathcal{B}_\eps \cap
\{y_1>0\}$ and $B_\eps (x)\cap U$. Consider on
$\mathcal{B}_\eps^+$ the standard Euclidean metric and
denote the corresponding $2$-dimensional Hausdorff measure 
with $\haus^2_e$. If
$\eps$ is sufficiently small, then $G_\eps$ can be chosen so
that the Lipschitz constants of $G_\eps$ and $G^{-1}_\eps$
are both smaller than $1+\eps$.
 Then, for any surface $\Delta\subset
B_\eps (x)\cap U$,
\begin{equation}\label{e:areaeuclidea}
(1-C\delta) \haus^2 (\Delta) \;\leq\; \haus^2_e (G_\eps
(\Delta)) \;\leq\; (1+C\delta) \haus^2 (\Delta)\, ,
\end{equation}
where $C$ is a universal constant.

We want to construct an
isotopy $\Lambda\in \Is (\mathcal{B}^+_\eps)$ such that $\Lambda (1,
\cdot) = G_\eps \circ \varphi (1, G_\eps^{-1}(\cdot))$ and
(for $k$ large enough)
\begin{equation}\label{e:areaeuclidea2}
\haus^2_e (\Lambda (t, G_\eps (\Delta^k)))\;\leq\; \haus^2_e (
G_\eps (\Delta^k))(1+C\delta) + C \delta \qquad \mbox{for every
$t\in [0,1]$.}
\end{equation}
After finding $\Lambda$, 
$\Phi (t, \cdot) = G_\eps^{-1}\circ \Lambda (t, G_\eps (\cdot))$
will be the desired map. Indeed $\Phi$ is an isotopy of $B_\eps (x)\cap
U$ which keeps a neighborhood of $B_\eps (x)\cap U$ fixed. It is
easily checked that $\Phi (1, \cdot) = \varphi (1, \cdot)$.
Moreover, by \eqref{e:areaeuclidea} and \eqref{e:areaeuclidea2}, for
$k$ sufficiently large we have
\begin{equation}\label{e:areavera}
\haus^2 (\Phi (t, \Delta^k)) \;\leq\; (1+C\delta) \haus^2 (\Delta^k)
+ C\delta \qquad \forall t\in [0,1]\, ,
\end{equation}
for some constant $C$ inpendent of $\delta$ and $k$. 
Since $\haus^2 (\Delta^k)$ is bounded by a constant independent of
$\delta$ and $k$, by choosing $\delta$ sufficiently small, we reach
the claim of the Lemma.

Next, we consider on $\mathcal{B}^+_\eps$ a one-parameter family of
diffeomorphisms. First of all we consider the continuous piecewise
linear map $\alpha :[0,1[\to [0,1]$ defined in the following way:
\begin{itemize}
\item $\alpha (t,s) = s$ for $(t+1)/2\leq s \leq 1$;
\item $\alpha (t,s) = (1-t) s$ for $0\leq s\leq t$;
\item $\alpha (t,s)$ is linear on $t\leq s \leq (t+1)/2$.
\end{itemize}
So, each $\alpha (t, \cdot)$ is a biLipschitz homeomorphism of
$[0,1]$ keeping a neighborhood of $1$ fixed, shrinking a portion of
$[0,1]$ and uniformly stretching the rest. For $t$ very close to
$1$, a large portion of $[0,1]$ is shrinked into a very small
neighborhood of $0$, whereas a small portion lying close to $1$ is
stretched to almost the whole interval.

Next, for any given $t\in [0,1[$, let $y_t:= ((1-t) \eta \eps, 0,0)$
where $\eta$ is a small parameter which will be fixed later. 
For any $z\in
\mathcal{B}^+_\eps$ we consider the point $\pi_t (z)\in
\partial\mathcal{B}^+_\eps$ such that the segment $[y_t, \pi_t (z)]$
containing $z$. We then define $\Psi (t,z)$ to be the point on the
segment $[y_t, \pi_t (z)]$ such that
$$
|y_t- \Psi (t,z)| = \alpha \left( t, \frac{|y_t-z|}{|x_t -\pi_t
(x)|}\right) |y_t- \pi_t (z)|\, .
$$
It turns out that $\Psi (0, \cdot)$ is the identity map and, for
fixed $t$, $\Psi (t,\cdot)$ is a biLipschitz homeomorphism of 
$\mathcal{B}^+_\eps$ keeping a neighborhood of $\partial 
\mathcal{B}^+_\eps$
fixed. Moreover, for $t$ close to $1$, $\Psi (t, \cdot)$ shrinks
a large portion of $\mathcal{B}^+_\eps$ in a neighborhood of $y_t$ and
stretches uniformly a layer close to $\partial \mathcal{B}_\eps$. 
See Figure \ref{f:shrink}.

We next consider the isotopy $\Xi (t, \cdot) :=
G_\eps^{-1} \circ \Psi (t, G_\eps (\cdot))$. It is easy to check
that, if we fix a $\Delta^k$ and we let $t\uparrow 1$, then the
surfaces $\Psi (1, G_\eps (\Delta^k))$ converge to the 
cone with center $0$
and base $G_\eps (\Delta^k)\cap \partial \mathcal{B}_\eps$.

\begin{figure}[htbp]
\begin{center}
    \input{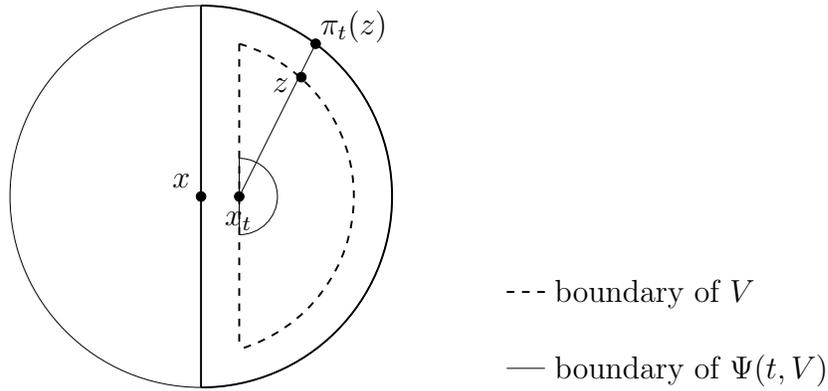}
    \caption{For $t$ close to $1$ the map $\Psi (t, \cdot)$
shrinks homotethically a large portion of $\mathcal{B}^+_\eps$.}
    \label{f:shrink}
\end{center}
\end{figure}

\subsection{Fixing a tangent cone}
By Subsection \ref{ss:tcones}, we 
can find a sequence $\rho_l\downarrow 0$
such that $V_{x, \rho_l}$ converges to a tangent cone $W$.
Our choice of the diffeomorphism $G_{\rho_l}$ implies that
$(O_{\rho_l}\circ G_{\rho_l})_{\#} V$ has the same varifold
limit as $V_{x, \rho_l}$.

Since $\Delta^k$ converges to $V$ in the sense of varifolds, by a
standard diagonal argument, we can find an increasing sequence of
integers $K_l$ such that:
\begin{itemize}
\item[(C)] $(O_{\rho_l} ( G_{\rho_l} (\Delta^{k_l}))$
converges in the varifold sense to $W$, whenever $k_l\geq K_l$.
\end{itemize}
(C), the conical property of $W$ and the coarea 
formula imply the following fact. For
$\rho_l$ sufficiently small, and for $k$ sufficiently large, there
is an $\eps\in ]\rho_l/2, \rho_l[$ such that:
\begin{equation}\label{e:stimacono}
\haus^2_e \big( \Psi (t, G_{\eps} (\Delta^k)\cap L)\big)
\;\leq\; \haus^2_e \big(G_{\eps} (\Delta^k)\cap L\big) +
\delta\, \qquad \mbox{$\forall t$ and all open $L\subset
\mathcal{B}^+_\eps$,}
\end{equation}
where $\Psi$ is the map constructed in the previous subsection. This
estimate holds independently of the small parameter $\eta$.
Moreover, it fixes the choice of $\eps_0$ and $K$ as in the
statement of the Lemma. $K$ depends only on the parameter $\delta$,
which will be fixed later. $\eps$ might depend on $k\geq K$, but it
is always larger than some fixed $\rho_l$, which will then be the
$\eps_0$ of the statement of the Lemma.

\medskip

\subsection{Construction of $\Lambda$.} Consider next the
isotopy $\psi = G_\eps \circ \varphi\circ G_\eps^{-1}$. By
definition, there exists a compact set $K$ such that $\psi (t, z)=z$
for $z\in \mathcal{B}^+_\eps\setminus K$ and every $t$. We now
choose $\eta$ so small that $K\subset \{x: x_1> \eta\eps\}$.
Finally, consider $T\in ]0,1[$ with $T$ sufficiently close to $1$.
We build the isotopy $\Lambda$ in the following way:
\begin{itemize}
\item for $t\in [0,1/3]$
we set $\Lambda (t, \cdot) = \Psi (3tT, \cdot)$;
\item for $t\in [1/3,2/3]$
we set $\Lambda (t, \cdot) = \Psi (3tT, \psi (3t -1, \cdot))$;
\item for $t\in [2/3,1]$ we set $\Lambda (t, \cdot) =
\Psi (3(1-t) T, \psi (1, \cdot))$.
\end{itemize}
If $T$ is sufficiently large, then $\Lambda$ satisfies
\eqref{e:areaeuclidea2}. Indeed, for $t\in [0,1/3]$,
\eqref{e:areaeuclidea2} follows from \eqref{e:stimacono}. Next,
consider $t\in [1/3,2/3]$. Since $\psi (t, \cdot)$ moves only points
of $K$, $\Lambda (t, x)$ coincides with $\Psi (T, x)$ except for $x$
in $\Psi (T, K)$. However, $\Psi (T,x)$ is homotethic to $K$ with a
very small shrinking factor. Therefore, if $T$ is chosen
sufficiently large, $\haus^2_e (\Lambda (t, G_\eps
(\Delta^k)))$ is arbitrarily close to $\haus^2_e (\Lambda (1/3,
G_\eps (\Delta^k)))$. Finally, for $t\in [2/3,1]$, $\Lambda
(t,x) = \Psi (3(1-t)T, x)$ for $x\not\in \Psi (3(1-t)T, K)$ and it
is $\Psi (3(1-t)T, \psi (1,x))$ otherwise. Therefore, $\Lambda (t,
G_\eps (\Delta^k))$ differs from $\Psi (3(1-t)T, G_\eps
(\Delta^k))$ for a portion which is a rescaled version of
$G_\eps (\varphi (1, \Delta^k)\setminus G_\eps
(\Delta^k)$. Since by hypothesis $\haus^2 (\varphi (1, \Delta^k))
\leq \haus^2 (\Delta^k)$, we actually get
$$
\haus^2_e \big(G_\eps (\varphi (1, \Delta^k))\setminus
G_\eps (\Delta^k)\big) \;\leq\; (1+C\delta) \haus^2_e
\big(G_\eps (\Delta^k) \setminus G_\eps (\varphi (1,
\Delta^k))\big)
$$
and by the scaling properties of the euclidean Hausdorff measure we
conclude \eqref{e:areaeuclidea2} for $t\in [2/3,1]$ as well.

Though $\Lambda$ is only a path of biLipschitz homeomorphisms, it is
easy to approximate it with a smooth isotopy: it suffices indeed to
smooth $\alpha|_{[0,T]\times [0,1]}$, for instance mollifying it
with a standard kernel.

\section{Proof of Proposition \ref{p:MSYbis}. Part III:
$\gamma$--reduction}\label{s:MSY3} In
this section we prove the following

\begin{lemma}[Interior regularity]\label{l:interior}
Let $V$ be as in Proposition \ref{p:MSYbis}.
Then $\|V\|= \haus^2\res \Delta$ where $\Delta$ is
a smooth stable minimal surface in $U$ (multiplicity is allowed).
\end{lemma}

In fact the lemma follows from the interior version of the squeezing
lemma and the following proposition, applying the regularity theory
of replacements as described in \cite{CD} (cp. with Section 7
therein).

\begin{propos}\label{MSYinterno} Let $U$ be an open ball
with sufficiently small radius. If $\Lambda$ is an
embedded surface with smooth boundary $\partial \Lambda \subset
\partial U$ and
  $\{\Lambda^k\}$ is a minimizing sequence for Problem $(\Lambda, \Is
  (U))$ converging to a varifold $W$,
  then there exists a stable minimal surface $\Gamma$ with
  $\overline{\Gamma}\setminus \Gamma \subset \partial \Lambda$ and
  $W=\Gamma$ in $U$.
\end{propos}

This Proposition has been claimed in \cite{CD} (cp. with Theorem 7.3
therein) and since nothing on the behavior of $W$ at the boundary is
claimed, it follows from a straightforward modification of the
theory of $\gamma$-reduction of \cite{MSY} (as asserted in
\cite{CD}). This simple modification of the $\gamma-$reduction is,
as the original $\gamma$-reduction, a procedure to reduce through
simple surgeries the minimizing sequence $\Lambda^k$ into a more
suitable sequence.

In this section we also wish to explain why this argument cannot be
directly applied neither to the surfaces $\Delta^k$ of Proposition
\ref{p:MSYbis} on the {\em whole} domain $U$ (see Remark \ref{r:nonsipuofare1}), 
nor to their
intersections with a smaller set $U'$ (see Remark \ref{r:nonsipuofare2}). 
In the first case, the
obstruction comes from the $1/j$-a.m. property, which is not
powerful enough to perform certain surgeries. In the second case
this obstruction could be removed by using the squeezing lemma, but
an extra difficulty pops out: the intersection $\Delta^k\cap
\partial U'$ is, this time, not fixed and the topology of
$\Delta^k\cap U'$ is not controlled. These technical problems are
responsible for most of the complications in our proof.

\subsection{Definition of the $\gamma$--reduction}
In what follows, we assume that an open set $U\subset M$ 
and a surface $\Lambda$ in $M$ with $\partial \Lambda \subset
\partial U$ are fixed. Moreover, we let
$\mathcal C$ denote the collection of all compact smooth
2-dimensional surfaces embedded in $U$ with boundary equal
to $\partial \Lambda$.

We next fix a positive number $\delta$ such that the conclusion of Lemma 1 in
\cite{MSY} holds and consider $\gamma<\delta^2/9$. Following
\cite{MSY} we define the $\gamma$-reduction and the
strong $\gamma$-reduction. 

\begin{definition}\label{d:gamma reduction}
For $\Sigma_1,\Sigma_2\in \mathcal{C}$
we write
$$
\Sigma_2\stackrel{(\gamma,U)}{\ll}\Sigma_1
$$
and we say that $\Sigma_2$ is a $(\gamma,U)-$reduction of
$\Sigma_1$, if the following conditions are satisfied:

\begin{itemize}
\item[($\gamma1$)] $\Sigma_2$ is obtained from $\Sigma_1$ through 
a surgery as described in Definition \ref{d:surgery}. Therefore:
\begin{itemize}
\item $\overline{\Sigma_1\setminus\Sigma_2}=A\subset U$ 
is diffeomorphic to the standard closed annulus
$\overline{\an (x, 1/2, 1)}$;
\item $\overline{\Sigma_2\setminus\Sigma_1}=D_1\cup 
D_2\subset U$ with each $D_i$ diffeomorphic to $\D$; 
\item There exists a set  $Y$ embedded in $U$, homeomorphic 
to $\B_1$ with $\partial Y=A\cup D_1\cup D_2$
and $(Y\setminus \partial Y)\cap(\Sigma_1\cup\Sigma_2)=\varnothing$.
(See Picture \ref{f:cutting}).
\end{itemize}
\item[($\gamma2$)]
$\haus^2(A)+\haus^2(D_1)+\haus^2(D_2)<2\gamma$;
\item[($\gamma3$)]If $\Gamma$ is the connected component of 
$\Sigma_1\cup \overline U$ containing $A$, then for each component
of $\Gamma\setminus A$ we have one of the following
possibilities:
\begin{itemize}
\item either it is a disc of area $\geq\delta^2/2$;
\item or it is not simply connected.
\end{itemize}
\end{itemize}
\end{definition}

\begin{remark}
The previous definition has another interesting consequence that the
reader could easily check: $\Sigma\in \mathcal C$ is
$(\gamma,U)-$irreducible if and only if whenever $\Delta$ is a disc
with $\partial \Delta=\Delta\cap\Sigma$ and
$\haus^2(\Delta)<\gamma$, then there is a disc $D\subset \Sigma$
with $\partial D=\partial \Delta$ and $\haus^2(D)<\delta^2/2$.
\end{remark}

A slightly weaker relation than
$\stackrel{(\gamma,U)}{\ll}$ can be defined as follows.
We consider $\Sigma_1,\Sigma_2\in\mathcal C$ and we say that
$\Sigma_2$ is a strong $(\gamma,U)-$reduction of $\Sigma_1$, written
$\Sigma_2\stackrel{(\gamma,U)}{<}\Sigma_1$, if there exists an
isotopy $\psi\in\Is (U)$ such that 
\begin{itemize}
\item[($s1$)]
$\Sigma_2\stackrel{(\gamma,U)}{\ll}\psi(\Sigma_1)$;
\item[($s2$)]
$\Sigma_2\cap(M\setminus U)=\Sigma_1\cap (M\setminus U)$;
\item[($s3$)]
$\haus^2(\psi(\Sigma_1)\triangle \Sigma_1)<\gamma$.
\end{itemize}
We say that $\Sigma\in \mathcal C$ is strongly
$(\gamma,U)-$irreducible if there is no $\tilde{\Sigma}\in \mathcal
C$ such that $\tilde{\Sigma}\stackrel{(\gamma,U)}{<}\Sigma$.

\begin{remark}\label{r:gamma1}
Arguing as in \cite{MSY} one
can prove that, for every $\Lambda'\in \mathcal{C}$,
there exist a constant $c\geq 1$ (depending on
$\delta$, $\gen(\Lambda')$ and $\haus^2(\Lambda')$) and a sequence
of surfaces $\Sigma_j$, $j=1,...,k$, such that
\begin{equation}\label{e:gamma1}
k\leq c;
\end{equation}
\begin{equation}\label{e:gamma2}
\Sigma_j\in \mathcal C;\quad j=1,...,k;
\end{equation}
\begin{equation}\label{e:gamma4}
\Sigma_k\stackrel{(\gamma,U)}{<}\Sigma_{k-1}\stackrel{(\gamma,U)}{<}...
\stackrel{(\gamma,U)}{<}\Sigma_1 = \Lambda'\, ;
\end{equation}
\begin{equation}\label{e:gamma5}
\haus^2 (\Sigma_k\Delta \Lambda')\;\leq\; 3 c \gamma\, ;
\end{equation}
\begin{equation}\label{e:gamma6}
\mbox{$\Sigma_k$ is strongly $(\gamma,U)-$irreducible.}
\end{equation}
Compare with Section 3 of \cite{MSY} and in particular with (3.3),
(3.4), (3.8) and (3.9) therein. 
\end{remark}

\subsection{Proof of Proposition \ref{MSYinterno}} 
Applying Lemma \ref{l:convexhull}, we conclude that
a susbsequence, not relabeled, of $\Lambda^k$
converges to a stationary varifold $V$ in $\overline{U}$
such that $\overline{U}\cap \supp (V) \subset \partial \Lambda$.
Next, arguing as in Section 6.1, we conclude that
$\|V\| (\partial \Lambda)=0$, and hence that $\|V\| (\partial U)=0$.
Arguing as in pages 364-365 of \cite{MSY} (see (3.22)--(3.26) therein),
we find a $\gamma_0>0$ and a sequence of $\gamma_0$--strongly irreducible
surfaces $\Sigma^k$ with the following properties:
\begin{itemize}
\item $\Sigma^k$ is obtained from $\Lambda^k$ through a number
of surgeries which can be bounded independently of $k$;
\item $\Sigma^k$ converges, in the sense of varifolds,
to $V$.
\end{itemize}
This allows to apply Theorem 2 and Section 5 of \cite{MSY} to
the surfaces $\Sigma^k$ to conclude that $\supp (V)\setminus \partial U$ 
is a smooth embedded stable minimal surface.  

\begin{remark}\label{r:nonsipuofare1}
This procedure {\em cannot} be applied if the minimality of
the sequence $\Lambda^k$ in
$\Is (U)$ were replaced by the minimality in $\Is_j (U)$.
In fact, the proof of Theorem 2 in \cite{MSY} uses heavily
the minimality in $\Is (U)$ and we do not know how to overcome
this issue. 
\end{remark}

\subsection{Proof of Lemma \ref{l:interior}}
Let $\Delta^k$ and $V$ be as in Proposition \ref{p:MSYbis}
and in Lemma \ref{l:interior}. Let $x\in U$ and consider
a $U'=B_\eps (x)\subset U$ as in Lemma \ref{l:squeeze}. 
Applying Lemma \ref{l:squeeze} we can
modify $\Delta^k$ in $B_\eps (x)$ getting a 
minimizing sequence $\{\Delta^{k,j}\}_j$ for $\Is (B_\eps (x))$.
Applying Proposition \ref{MSYinterno}, we can
assume that $\Delta^{k,j}$ converges, as $j \uparrow\infty$
to a varifold $V'_k$ which in $B_\eps (x)$ is
a stable minimal surface $\Sigma^k$. 
By the curvature estimates for minimal surfaces
(cp. also with the Choi-Schoen compactness Theorem), we 
can assume that $\Sigma^k$ converges to a stable
smooth minimal surface $\Sigma^\infty$. 
Extracting a diagonal subsequence $\tilde{\Delta}^k := \Delta^{k, j(k)}$,
we can assume that $\tilde{\Delta}^k$ is still minimizing
for problem $\Is_j (U)$ and hence that it converges
to a varifold $V'$. $V'$ coincides with $\Sigma$ in $B_\eps (x)$
and with $V$ outside $B_\eps (x)$ and hence it is a replacement
according to Definition 6.2 in \cite{CD} (see Section 7 therein).
By Proposition 6.3 of \cite{CD}, $V$ coincides with a smooth embedded
minimal surface in $U$. 

\begin{remark}\label{r:nonsipuofare2}
Note that the arguments of Section 3 of \cite{MSY}
{\em cannot} be applied directly to the sequence $\Delta^k$.
It is indeed possible to modify $\Delta^k$ in $B_\eps (x) =:U'$
to a strongly $\gamma$-irreducible $\tilde{\Delta}^k$. However,
the number of surgeries needed is controlled by $\haus^2 (\Delta^k
\cap B_\eps (x))$ and $\gen (\Delta^k\cap U')$.
Though the first quantity can be bounded independently of $k$,
on the second quantity (i.e. $\gen (\Delta^k\cap U')$) we
do not have any a priori uniform bound. 
\end{remark}

\section{Proof of Proposition \ref{p:MSYbis}. Part IV:
Boundary regularity.}\label{s:MSY4}

In this section we conclude the proof of the first part of
Propositions \ref{p:MSYbis} and \ref{p:MSYtris}. More precisely, we
show that the surface $\Delta$ of Lemma \ref{l:interior} is regular
up to the boundary and its boundary coincides with $\partial
\Sigma$.

\begin{lemma}[Boundary regularity]\label{l:boundary}
Let $\Delta$ be as in Lemma \ref{l:interior}. Then $\Delta$ has a
smooth boundary and $\partial \Delta =\partial \Sigma$.
\end{lemma}

As a corollary, we conclude that the multiplicity of $\Delta$ is
everywhere $1$.

\begin{corol}\label{c:multiplicity}
There exist finitely many stable embedded connected disjoint minimal
surfaces $\Gamma_1, \ldots, \Gamma_N\subset U$ with disjoint smooth
boundaries and with multiplicity $1$ such that
\begin{equation}\label{e:split}
\Delta \;=\; \Gamma_1\cup\ldots \cup\Gamma_N \qquad\mbox{and}\qquad
\partial \Delta \;=\;
\partial \Gamma_1\cup\ldots\cup\partial \Gamma_N\, .
\end{equation}
\end{corol}

\begin{proof} Lemmas \ref{l:interior} and \ref{l:boundary}
imply that $\Delta$ is the union of finitely many disjoint connected
components $\Gamma_1\cup\ldots \cup\Gamma_N$ contained in $U$
and that either $\partial \Gamma_i = 0$ or $\partial \Gamma_i$ is 
the union of some connected components of $\partial \Sigma$.
In this last case, the multiplicity of $\Gamma_i$ is necessarily $1$.
On the other hand, $\partial \Gamma_i = 0$ cannot occur,
otherwise $\Gamma_i$ would be a smooth embedded minimal surface
without boundary contained in a convex ball of a Riemannian manifold,
contradicting the classical maximum principle.
\end{proof}

\subsection{Tangent cones at the boundary}\label{ss:tcones2}
Consider now $x\in \supp \|V\|\cap \partial U$.
We follow Subsection \ref{ss:tcones} and consider the chart 
$f_{x,\rho} : B_\rho (x)\to \mathcal{B}_1$ given by
$f_{x,\rho} (y) = \exp_x^{-1} (y)/\rho$. 
We then denote by $V_{x,\rho}$ the varifolds $(f_{x,\rho})_\# V$.
Moreover, if $\lambda>0$, we will denote by $O_\lambda:
\rn{3}\to \rn{3}$ the rescaling $O_\lambda (x) = x/\lambda$.

Let next $W$ be the limit of a subsequence $V_{x, \rho_j}$. Again
following the discussion of Subsection \ref{ss:tcones}, we can
choose a system of coordinates $(y_1, y_2, y_3)$ such that:
\begin{itemize}
\item $W$ is integer rectifiable and $\supp (W)$ is contained in the wedge
$$
\Wedge 
\;:=\; \{(y_1, y_2, y_3) : |y_2|\;\leq\; y_1 \tan \theta_0\} \cap
\overline{\mathcal{B}}_1 (0)\, .
$$
\item $\supp (W)$ containes the line $\ell = \{(0,0,t) : t\in [-1,1]\}$,
(which is the limit of the curves $f_{x, \rho} 
(\partial \Sigma\cap B_\rho (x))$).
\item  If we denote by $r: \rn{3}\to \rn{3}$ the reflection given by
$r (z_1, z_2, z_3) = (-z_1, -z_2, z_3)$, then $r_\# W + W$
is a stationary cone.
\end{itemize} 

By the Boundary regularity Theorem of Allard (see Section 4 of \cite{All2}),
in order to show regularity it suffices to prove that
\begin{itemize}
\item[(TC)] Any $W$ as above (i.e. any varifold limit of a subsequence
$({f_x^{\rho_n}})_\# V$ with $\rho_n\downarrow 0$) is 
a half--disk of the form
\begin{equation}\label{half-disks}
P_\theta := \{(y_1, y_2, y_3): y_1>0 , y_3 = y_1 \tan \theta\}\cap
\mathcal{B}_1 (0)
\end{equation}
for some angle $\theta\in ]-\pi/2, \pi/2[$.
\end{itemize}

\medskip

In the rest of this section we aim, therefore, at proving (TC).
As a first step we now show that  
\begin{equation}\label{e:un_of_disks}
W \;=\; \sum_{i=1}^N k_i P_{\theta_i}\,
\end{equation}
where $k_i\geq 1$ are integers and $\theta_i$ are angles in
$[-\theta_0, \theta_0]$. There are two possible ways of seeing this.
One way is to use the Classification of stationary integral varifolds
proved by Allard and Almgren in \cite{AA}. 

The second, which is perhaps
simpler, is to observe that, on $\mathcal{B}^+$ the varifold $W$ is actually
smooth. Indeed, by the interior regularity, $V$ is a smooth minimal
surface in $B_\rho (x)\cap V$ and it is stable, therefore, by Schoen's
curvature estimates, a subsequence of $V_{x, \rho_n}$ converges 
smoothly in compact subsets of $\mathcal{B}^+$.
It follows that $W^r := W+ r_\# W$ coincides with a smooth minimal
surface outside on $\mathcal{B}_1 (0)\setminus \ell$. On the other
hand $W^r $ is a cone and therefore we conclude that $\partial \mathcal{B}_{1/2}
(0)\cap W^r \setminus \{(0,0,1/2), (0,0,-1/2)\}$ is a smooth $1$-d manifold
consisting of arcs of great circles. Since $\supp (W)\subset \Wedge$,
we conclude that in fact $\partial \mathcal{B}_{1/2}
(0)\cap W^r \setminus \{(0,0,1/2), (0,0,-1/2)\}$ consists
of finitely many planes (mupltiplicity is allowed) passing through
$\ell$. This proves \eqref{e:un_of_disks}. 

\subsection{Diagonal sequence}
We are now left with the task of showing that $N=1$ 
and $k_1=1$. We will, indeed,
assume the contrary and derive a contradiction. 
In order to do so,
we consider a suitable diagonal sequence $f_{x,\rho_n}
(\Delta^{k_n})$ converging, in the sense of varifolds, to $W$. 
We can select $\Delta^{k_n}$ in such a way that the following
minimality property holds:
\begin{itemize}
\item[(F)] If $\Lambda$ is any surface isotopic to
$\Delta^{k_n}$ with an isotopy fixing $\partial (U\cap B_{\rho_n}
(x))$, then $\haus^2 (\Lambda)\geq \haus^2 (\Delta^{k_n})
- \rho_n^3$.
\end{itemize}
Indeed, we appply the Squeezing Lemma \ref{l:squeeze} with $\beta =
1/(16j)$ and let $n$ be so large that $\rho_n$ is smaller
than the constant $\eps_0$ given by the Lemma. Since $\Delta^k$
is $1/j$--a.m. in $U$, we conclude
therefore that, if we set
$$
M_{k,n} \;:=\; \inf \{\Phi (1, \Delta^k): \Phi\in \Is (U\cap
B_{\rho_n} (x))\}\, ,
$$
then 
$$
\lim_{k\uparrow\infty}
\haus^2 (\Delta^k\cap B_{\rho_n} (x)) - M_{n,k} \;=\; 0\, .
$$
Therefore, having fixed $\rho_n <\eps_0$, we can choose 
$k_n$ so large that $M_{n,k}\geq \haus^2 (\Delta^{k_n}) - \rho_n^3$.

Next, it is convenient to introduce a slightly perturbed
chart $g_x^{\rho_n}$ which maps $\partial U \cap B_{\rho_n} (x)$
onto $\mathcal{B}_1 \cap \{y_1=0\}$ and $\partial \Sigma
\cap B_{\rho_n (x)}$ onto $\ell$. This can be done
in such a way that $f_{x, \rho_n} \circ g_{x, \rho_n}^{-1}$
and $g_{x, \rho_n}\circ f_{x, \rho_n}^{-1}$ converge
smoothly to the identity map as $\rho_n\downarrow 0$.

Having set $\Gamma_n = g_{x, \rho_n} (\Delta^{k_n})$,
we have that $\Gamma_n$ converges to $W$ in the sense of varifolds.
Moreover, our discussion implies that
$\haus^2 (\Delta^{k_n}\cap B_{\rho_n} (x))
\;=\; \rho_n ^2 \haus^2_e (\Gamma_n) + O (\rho_n^3)$. 
Therefore we conclude from (F) that
\begin{itemize}
\item[(F')] Let $m_n$ be the minimum of
$\haus^2_e (\Lambda)$ over all surfaces $\Lambda$
isotopic to $\Gamma_n$ with an isotopy which fixes 
$\partial (U\cap \mathcal{B}_1)$. 
Then $\haus^2_e (\Gamma_n) - m_n \downarrow 0$.
\end{itemize}

We next claim that 
\begin{equation}\label{e:sottosucc1}
\liminf_{n\downarrow 0}
\haus^1_e (\Gamma_n \cap \partial \mathcal{B}_\sigma)
\;\geq\; \pi \sigma \sum_{i=1}^N k_i \qquad
\mbox{for every $\sigma \in ]0,1[$.}
\end{equation}
Indeed, using the squeezing homotopies
of Section \ref{ss:squeeze} it is easy to see 
that 
$$
\haus^2_e (\Gamma_n) - m_n \;\geq\; 
\haus^2_e (\Gamma_n \cap \mathcal{B}_\sigma) 
- \sigma \haus^1_e (\Gamma_n \cap \partial \mathcal{B}_\sigma)
$$
Letting $n\uparrow 0$ and using \eqref{e:un_of_disks} with the convergence of $\Gamma_n$
to the varifold $W$ we conclude
$$
\liminf_{n\uparrow \infty} (\haus^2_e (\Gamma_n) - m_n) \;\geq\;
\sigma \left(\sigma \pi \sum_i k_i - 
\liminf_{n\downarrow 0} \haus^1_e 
(\Gamma_n \cap \partial \mathcal{B}_\sigma)\right)\, .
$$
Therefore, from (F') we conclude \eqref{e:sottosucc1}.

We next claim the existence of a $\sigma\in [1/2,1[$
and a subsequence $n(j)$ such that $\Gamma_{n(j)}\cap
\partial \mathcal{B}_\sigma$ is a smooth $1$-dimensional
manifold with boundary $(0,0, \sigma)- (0,0, -\sigma)$
and, at the same time,
\begin{equation}\label{e:sottosucc2}
\lim_{j\uparrow \infty}
\haus^1_e (\Gamma_{n(j)} \cap \partial \mathcal{B}_\sigma)
\;=\; \pi \sigma \sum_{i=1}^N k_i 
\end{equation}
and 
\begin{equation}\label{e:sottosucc3}
\lim_{j\uparrow \infty}
\haus^1_e (\Gamma_{n(j)} \cap \partial \mathcal{B}_\sigma \setminus K)
\;=\; 0 \qquad \mbox{for every compact $K\subset \mathcal{B}_1
\setminus \bigcup_i P_{\theta_i}$.}
\end{equation}

In fact, let $\{K_l\}_l$ be an exhaustion of $\mathcal{B}_1
\setminus \bigcup_i P_{\theta_i}$ by compact sets.
Observe that, by the convergence of $\Gamma_n$ to $W$,
we get
\begin{equation}\label{e:sommatoria}
\lim_{n\uparrow\infty}
\left(\haus^2_e (\Gamma_n\cap \mathcal{B}_1
\setminus \mathcal{B}_{1/2}) + \sum_{l=0}^\infty
2^{-l} \haus^2_e (\Gamma_n \setminus K_l \cap (\mathcal{B}_1
\setminus \mathcal{B}_{1/2}))\right) \;=\;
\frac{\pi}{8} \sum_i k_i\, .
\end{equation}
Using the coarea formula, we conclude
$$
\int_{1/2}^1 \sigma \pi \sum_i k_i\, d\sigma
\;\geq\; \lim_{n\uparrow\infty} 
\int_{1/2}^1 \left(\haus^1_e (\Gamma_n \cap
\partial \mathcal{B}_\sigma) +
\sum_l 2^{-l} \haus^1_e (\Gamma_n \cap
\partial \mathcal{B}_\sigma\setminus K_l) \right)\, d\sigma\, .
$$
Therefore, by Fatou's Lemma, for a.e. $\sigma\in [1/2,1[$
there is a subsequence $n(j)$ such that
\begin{equation}\label{e:sottosucc4}
\lim_{j\uparrow \infty}
\left(\haus^1_e (\Gamma_n \cap
\partial \mathcal{B}_\sigma) +
\sum_l 2^{-l} \haus^1_e (\Gamma_n \cap
\partial \mathcal{B}_\sigma\setminus K_l) \right) \;=\; \pi
\sigma \sum_{i} k_i \, .
\end{equation}
Clearly, \eqref{e:sottosucc1} and \eqref{e:sottosucc4}
imply \eqref{e:sottosucc2} and \eqref{e:sottosucc3}.
On the other hand, by Sard's Theorem, for a.e. $\sigma\in [1/2, 1[$
every surface $\partial \mathcal{B}_\sigma \cap \Gamma_n$
is a smooth $1$-dimensional submanifold with boundary
$(0,0,\sigma)-(0,0,-\sigma)$.  

\subsection{Disks}\label{ss:disks}
From now on we fix the radius $\sigma$ found above and
we use $\Gamma_n$ in place of $\Gamma_{n(i)}$ (i.e. we do
not relabel the subsequence).
Consider now the Jordan curves $\gamma^n_1, \ldots, \gamma^n_{M(n)}$
forming $\Gamma^n\cap \partial \mathcal{B}^+_\sigma$
(by $\mathcal{B}^+_\sigma$ we understand the half ball
$\mathcal{B}_\sigma \cap \{y_1\geq 0\}$). 

Since $\partial \Gamma^n \cap \{y_1 =0\}$ is given by the
segment $\ell$, there is one curve, say $\gamma^n_1$, which
contains the segment $\ell$. All the others, i.e. the curves
$\gamma^n_i$ with $i\geq 2$ lie in $\partial \mathcal{B}_\sigma
\cap \{y_1>0\}$. 

\begin{figure}[htbp]
\begin{center}
    \input{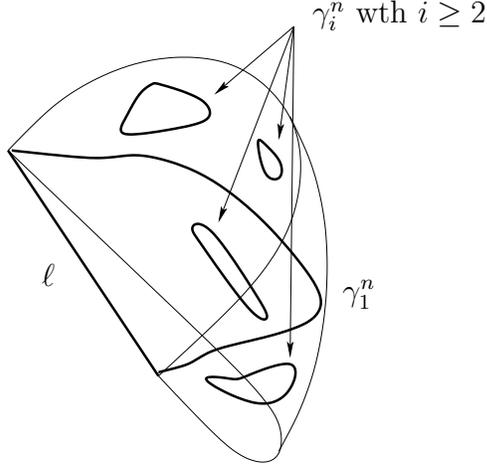}
    \caption{The curves $\gamma^n_i$.}
    \label{f:gira_nongira}
\end{center}
\end{figure}

Next, for every $\gamma^n_l$ consider the number
\begin{equation}\label{e:kappa}
\kappa^n_l \;:=\; 
\inf \{ \haus^2_e (D): \mbox{$D$ is an embedded smooth disk bounding
$\gamma^n_l$}\}\, .
\end{equation}

We will split our proof into several steps. 

\begin{itemize}
\item[(a)]
In the first step, we combine a simple desingularization procedure
with the fundamental result of Almgren and Simon (see \cite{AS}),
to show that
$$
\mbox{there are disjoint embedded
smooth disks $D^n_1, \ldots D^n_{M(n)}$ s.t.}
$$
\begin{equation}\label{e:claim(a)}
\qquad \sum_{i=1}^{M(n)} \haus^2_e (D^n_i)\;\leq\; 
\sum_{i=1}^{M(n)} \kappa^n_i + \frac{1}{n}\, .
\end{equation}
A simple topological observation (see Lemma
\ref{l:technical_topological} in the Appendix \ref{a:isotopie})
shows that, for each fixed $n$, 
there exist isotopies $\Phi_l$
keeping $\partial \mathcal{B}^+_\sigma$ fixed and such that $\Phi_l
(\Gamma_n \cap \mathcal{B}_\sigma)$ converges, in the sense of
varifolds, to the union of the disks $D^n_i$. Combining 
(F'), \eqref{e:claim(a)} and the convergence of
$\Gamma_n$ to the varifold $W$ we then conclude
\begin{equation}\label{e:(a')}
\limsup_{n\uparrow\infty} \sum_{i=1}^{M(n)} \kappa^n_i
\;=\; \pi \sigma^2 \sum_j k_j\, .
\end{equation} 

\item[(b)] In the second step we will show the
existence of a $\delta>0$ (independent of $n$) such that
\begin{equation}\label{e:claim(b)}
\kappa^n_i \;\leq\; \sigma \left(\frac{1}{2}-\delta\right)
\haus^1_e (\gamma^n_i)\qquad \mbox{for every $i\geq 2$ and every $n$.}
\end{equation} 
A simple cone construction shows that
\begin{equation}\label{e:(b')}
\kappa^n_1 \;\leq\; \frac{\sigma}{2} \haus^1_e (\gamma^n_1)\, .
\end{equation}
So, \eqref{e:sottosucc2}, \eqref{e:claim(b)} and
\eqref{e:(b')} imply 
\begin{equation}\label{e:(b'')}
\lim_{n\uparrow \infty} \sum_{i=2}^{M(n)}
\haus^1_e (\gamma^n_i)\;=\; 0\qquad
\mbox{and} \qquad \lim_{n\uparrow \infty} 
\haus^1_e (\gamma^n_1) \;=\; \sigma \sum_j k_j\, ,
\end{equation}
which in turn give
\begin{equation}\label{e:(b''')}
\lim_{n\uparrow\infty} \kappa^n_1 \;=\; \frac{\pi \sigma^2}{2} 
\sum_j k_j\, .
\end{equation}

\item[(c)] We next fix a parameterization $\beta^n_1 : \SS^1 \to 
\partial \mathcal{B}^+_\sigma$ of $\gamma^n_1$ with a multiple
of the arc-length and extract a further
subsequence, not relabeled such that $\beta^n_1$ converges
to a $\beta^\infty$. By \eqref{e:sottosucc3}, the image
of $\beta^\infty$ is then contained in the union of the curves
$P_{\theta_l}\cap \partial \mathcal{B}^+_\sigma$. We will then
show that
\begin{equation}\label{e:claim(c)}
\limsup_{n\downarrow\infty} \kappa^n_1 \;=\; \frac{\pi \sigma^2}{2}\, .
\end{equation}
\eqref{e:(b''')} and \eqref{e:claim(c)} finally show that
$W$ consists of a single half-disk $P_\theta \cap \mathcal{B}^+_1$,
counted once. This will therefore complete the proof.
\end{itemize}

\subsection{Proof of \eqref{e:claim(a)}}
In this step we fix $n$ and prove the claim \eqref{e:claim(a)}.
First of all, note that each $\gamma_i^n$ with $i\geq 2$ is a smooth Jordan
curve lying in $\partial B_\sigma\cap \{y_1>0\}$.

We recall the following result of Almgren and Simon (see \cite{AS}). 

\begin{theorem}\label{t:AlmgrenSimon}
For every curve $\gamma^n_i$ with $i\geq 2$ consider
a sequence of smooth disks $D^j$ with $\haus^2_e (D^j)$ converging
to $\kappa^n_i$. Then a subsequence, not relabeled,
converges, in the sense of varifolds, to an 
embedded smooth disk $D^n_i\subset \mathcal{B}^+_\sigma$ 
bounding $\gamma^n_i$ and such that $\haus^2_e (D^n_i) = \kappa^n_i$.
(The disk is smooth also at the boundary).
\end{theorem}

For each $\gamma^n_i$ select therefore a disk $D^n_i$ as 
in Theorem \ref{t:AlmgrenSimon}.
We next claim that these disks are all pairwise disjoint. Fix
in fact two such disks. To simplify the notation
we call them $D^1$ and $D^2$ and assume they bound, respectively,
the curves $\gamma_1$ and $\gamma_2$. Clearly, $D^1$ divides
$\mathcal{B}^+_\sigma$ into two connected components $A$ and $B$
and $\gamma_2$ lies in one of them, say $A$. We will
show that $D^2$ lies in $A$. 

Assume by contradiction that $D^2$ intersects $D^1$. 
By perturbing $D^2$ a little
we modify it to a new disk $E^j$ such that $\haus^2_e (E^j)\leq
\haus^2_e (D^2)+1/j$ and $E^j$ intersects $D^1$ transversally
in finitely many smooth Jordan curves $\alpha_m$. 

Each $\alpha_m$ bounds a disk $F^m$ in $E^j$. We call
$\alpha_m$ maximal if it is not contained in any $F^l$. 
Each maximal $\alpha_m$ bounds also a disk $G^m$ in
$D^1$. By the minimality of $D^1$, clearly $\haus^2_e (G^m)
\leq \haus^2_e (F^m)$. We therefore consider 
the new disk $H^j$ given by 
$$
D^2\setminus \left(\;\bigcup_{\alpha_m \mbox{\;\tiny maximal}} 
F^m\right) \cup \bigcup_{\alpha_m \mbox{\;\tiny maximal}} G^m\, .
$$
Clearly $\haus^2_e (H^j)\leq \haus^2_e (E^j) +1/j$. With 
a small perturbation we find a nearby smooth embedded disk
$K^j$ which lies in $A$ and has 
$\haus^2_e (K^j)\leq \haus^2_e (E^j)+1/(2j)$. 
By letting $j\uparrow \infty$ and applying
Theorem \ref{t:AlmgrenSimon}, a subsequence of $K^j$ converges
to a smooth embedded minimal disk $D^3$ in the sense of varifolds.
On the other hand, by choosing $K^j$ sufficiently close
to $H^j$, we conclude that $H^j$ converges as well to
the same varifold. But then, 
$$
D^2\setminus \left(\;\bigcup_{\alpha_m \mbox{\;\tiny maximal}} 
F^m\right) \;\subset\; D^3
$$
and hence $D^2\cap D^3 \neq \emptyset$. Since $D^3$ lies
on one side of $D^2$ (i.e. in $\overline{A}$) this violates
the maximum principle for minimal surfaces.

\medskip

Having chosen $D^n_2, \ldots D^n_{M(n)}$ as above,
we now choose a smooth disk $E^n_1$ bounding $\gamma^n_1$
and with 
$$
\haus^2_e (E^n_1)\;\leq\; \kappa^n_1 + \frac{1}{3n}\, .
$$
In fact we cannot apply directly Theorem \ref{t:AlmgrenSimon}
since in this case the curve $\gamma^n_1$ is not smooth but
has, in fact, two corners at the points $(0,0,\sigma)$
and $(0,0,-\sigma)$. 

$\gamma^n_1$ lies in one connected component $A$ of
$\overline{\mathcal{B}^+_\sigma}$. We now find a new smooth embedded
disk $D^n_1$ with
$$
\haus^2_e (D^n_1) \;\leq\; \kappa^n_1 + \frac{1}{n}\,
$$
and lying in the interior of $A$
This suffices to prove \eqref{e:claim(a)}.

Consider the disks $D'_1, \ldots D'_l$
which, among the $D^n_j$ with $j\geq 2$, bound $A$. 
We first perturb $E^n_1$ to a smooth embedded
$F^n_1$ which intersects all the $D'_j$. We then
inductively modify $E^n_1$ to a new disk which does not
intersect $D'_j$ and looses at most $1/(3ln)$ in area.
This is done exactly with the procedure outlined above
and since the distance between different $D'_j$'s is always positive,
it is clear that while removing the intersections
with $D'_j$ we can do it in such a way that we do not add
intersections with $D'_i$ for $i<j$.

\subsection{Proof of \eqref{e:claim(b)}} In this
step we show the existence of a positive $\delta$, independent
of $n$ and $j$, such that
\begin{equation}\label{e:better_than_cone}
\kappa^n_j \;\leq\; \sigma \left(\frac{1}{2}-\delta\right)
\haus^1_e (\gamma^n_j) \qquad \forall j\geq 2, \forall n\, .
\end{equation}
Observe that for each $\gamma^n_j$ we can construct the cone
with vertex the origin, which is topologically a disk
and achieves area equal to $\textstyle{\frac{\sigma}{2}}
\haus^1_e (\gamma^n_j)$. On the other hand, this cone is clearly
not stationary, because $\gamma^n_j$ is not a circle,
and therefore there is a disk diffeomorphic to the cone
with area strictly smaller than $\textstyle{\frac{\sigma}{2}}
\haus^1_e (\gamma^n_j)$. A small perturbation of this disk yields
a smooth embedded disk $D$ bounding $\gamma^n_j$ such that
\begin{equation}
\haus^2_e (D) \;<\; \frac{\sigma}{2} \haus^1_e (\gamma^n_j)\, .
\end{equation}
Therefore, it is clear that it suffices
to prove \eqref{e:better_than_cone} when $n$ is large enough. 

Next, by the isoperimetric inequality, there is
a constant $C$ such that, any curve $\gamma$ in 
$\partial \mathcal{B}_\sigma$ bounds, in $\mathcal{B}_\sigma$,
a disk $D$ such that
\begin{equation}\label{e:isoperimetrica}
\haus^2_e (D) \;\leq\; C \left(\haus^1_e (\gamma)\right)^2\, .
\end{equation}
Therefore, \eqref{e:better_than_cone} is clear for
every $\gamma^n_j$ with $\haus^1_e (\gamma^n_j) \leq \sigma/4C$.

We conclude that the only way of violating 
\eqref{e:better_than_cone} is to have a subsequence,
not relabeled, of curves $\gamma^n := \gamma^n_{j(n)}$ such that
\begin{itemize}
\item $\haus^1_e (\gamma^n)$ converges to some
constant $c_0>0$;
\item $\kappa^n:= \kappa^n_{j(n)}$ converges to $c_0/2$.
\end{itemize}
Consider next the wedge $\Wedge= \{ |y_2|\leq y_1\tan \theta_0\}$ 
containing the support of the varifold $V$. If we enlarge this
wedge slightly to
$$
\Wedge' \;:=\; \{|y_2|\leq y_1(\tan \theta_0 +1)\}\, ,
$$
we conclude, by \eqref{e:sottosucc3}, that
\begin{equation}\label{e:out}
\lim_{n\uparrow\infty} \haus^1_e (\gamma^n\setminus \Wedge') \;=\; 0\, .
\end{equation}

\medskip

Perturbing $\gamma^n$ slightly we find a nearby smooth 
Jordan curve $\beta^n$ contained in $\partial \mathcal{B}_\sigma
\cap \Wedge'$. Consider next 
\begin{equation}\label{e:minimomu}
\mu^n \;:=\; \min \{ \haus^2_e (D) : \mbox{smooth embedded disk $D$
bounding $\beta^n$}\}\, .
\end{equation}
Given a $D$ bounding $\beta^n$, it is possible to construct
a $D'$ bounding $\gamma^n$ with 
$$\haus^2_e (D')\leq\haus^2_e (D) + o(1).$$

Therefore, we conclude that
\begin{itemize}
\item $\haus^1_e (\beta^n)$ converges to $c_0>0$;
\item $\mu^n$ converges to $\sigma c_0/2$;
\item $\beta^n$ is contained in $\Wedge'$. 
\end{itemize}

\medskip

Consider next the projection of the curve $\alpha = \Wedge'\cap 
\mathcal{B}_\sigma$ on the plane $\pi = y_1y_3$. This projection
is an ellypse bounding a domain $\Omega$ in $\pi$.
Clearly $\alpha$ is the graph of a function
over this ellypse. The function is Lipschitz (actually it is 
analytic except for the two points $(0, \sigma)$ and $(0, -\sigma)$)
and we can therefore find a function $f: \Omega \to \R$ 
which minimizes the area of its graph. This function
is smooth up to the boundary except in the points $(0, \sigma)$
and $(0, -\sigma)$ where, however, it is continuous.
Therefore, the graph of $f$ is an embedded disk. 

We denote by $\Lambda$ the graph of $f$. $\Lambda$
is in fact the {\em unique} area-minimizing current spanning
$\alpha$, by a well-known property of area--minimizing graphs. 
By the classical maximum
principle, $\Lambda$ is contained in the wedge $\Wedge'$ and does
not contain the origin. Consider next the cone $C^n$
having vertex in $0$ and $\beta^n$ as base. Clearly, this cone
intersects $\Lambda$ in a smooth Jordan curve $\tilde{\beta}^n$ and
hence there is a disk $D^n$ in $\Lambda$ bounding this curve.
Moreover, we call $E^n$ the cone constructed on
$\tilde{\beta}^n$ with vertex $0$ (see Figure \ref{f:wedgeminimo}).

\begin{figure}[htbp]
\begin{center}
    \input{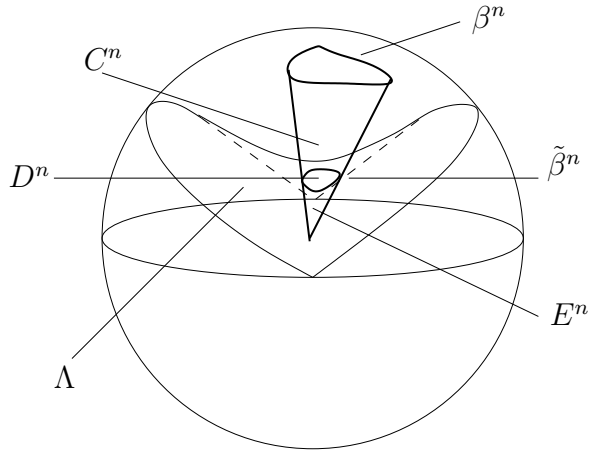}
    \caption{The minimal surface $\Lambda$, the cones
$C^n$ and $E^n$ and the domain $D^n$.}
    \label{f:wedgeminimo}
\end{center}
\end{figure}

Clearly, 
\begin{equation}\label{e:positivo}
\liminf_{n\uparrow \infty} \haus^1_e({\beta}^n) \;>\;0\, .
\end{equation}
Consider next the current given by $D^n \cup (C^n\setminus E^n)$.
These coverge, up to subsequences, to some integer rectifiable current.
Therefore,
the disks $D^n$ converge, in the sense of currents, to a 
$2$--dimensional current $D$ supported in $\Lambda$. It is
easy to check that $D$ must be the current represented
by a domain of $\Lambda$, counted with multiplicity $1$.
Therefore
\begin{equation}\label{e:convergenza_area}
\lim_{n\uparrow\infty} \haus^2_e (D^n) \;=\; \haus^2_e (D)\, .
\end{equation}
Similarly,
$E^n$ converges, up to subsequences, to a current $E$.
By the minimizing
property of $\Lambda$, $\haus^2_e (D) < {\bf M} (E)$,
unless $\haus^2_e (D)= {\bf M} (E) = 0$, where ${\bf M} (E)$
denotes the mass of $E$. 

So, if ${\bf M} (E) > 0$, we then have
$$
\liminf_{n\uparrow \infty} \haus^2_e (E^n)\;\geq\;
{\bf M} (E) \;>\; \haus^2_e (D) \;=\;
\lim_{n\uparrow \infty} \haus^2_e (D^n)\, .
$$
If ${\bf M} (E) =0$, by \eqref{e:positivo}, we conclude
$$
\liminf_{n\uparrow \infty} \haus^2_e (E^n)
\;>\; 0 \;=\; \lim_{n\uparrow \infty} \haus^2_e (D^n)\, .
$$
In both cases we conclude that the embedded disk
$H^n = (C^n\setminus E^n) \cup D^n$ bounds $\beta^n$ and satisfies
\begin{equation}\label{e:dimeno}
\lim_{n\uparrow\infty} \haus^2_e (H^n) \;<\;
\lim_{n\uparrow\infty} \haus^2_e (C^n) \;=\; \frac{\sigma c_0}{2}
\;=\; \lim_{n\uparrow \infty} \mu^n\, .
\end{equation}
Therefore, there exists an $n$ such that $\mu^n > \haus^2_e (H^n)$.
A small perturbation of $H^n$ gives a smooth embedded 
disk bounding $\beta^n$
with area strictly smaller than $\mu^n$. This contradicts
the minimality of $\mu^n$ (see \eqref{e:minimomu})
and hence proves our claim.

\subsection{Proof of \eqref{e:claim(c)}} 
In this final step we show \eqref{e:claim(c)}. Our arguments
are inspired by those of Section 7 in \cite{AS}.

\medskip

Consider the
curve $\gamma^n_1$. Again applying \eqref{e:sottosucc3}
we conclude that, for every compact set 
$$
K\subset \overline{\mathcal{B}}^+_\sigma\setminus \bigcup_i 
P_{\theta_i}
$$
we have
\begin{equation}\label{e:converge}
\lim_{n\uparrow \infty} \haus^1_e (\gamma^n_1\setminus K) \;=\; 0\, .
\end{equation}

Consider next the solid sector $S := \Wedge'\cap
\mathcal{B}_\sigma$. Clearly $\haus^2_e (\partial S) = 
(3\pi -\eta) {\sigma}^2$,
where $\eta$ is a positive constant. 
Clearly a curve contained in $\partial S$ bounds
always a disk with area at most $\pi (\textstyle{\frac{3}{2}}-
\textstyle{\frac{\eta}{2}}) \sigma^2$. 
For large $\gamma^n_1$ we can modify it to a new curve 
$\tilde{\gamma}^n$ contained in $\partial S$, and hence 
find a smooth embedded disk bounding $\tilde{\gamma}^n$ with
area at most $\pi (\textstyle{\frac{3}{2}}
- \textstyle{\frac{\eta}{4}}) 
\sigma^2$. This and \eqref{e:(b''')} implies that
$$
\frac{\pi \sigma^2}{2} \sum_i k_i \;=\; \lim_{n\uparrow\infty}
\kappa^n_1 \;<\; \frac{3\pi}{2} \sigma^2\, .
$$
Therefore we conclude that $\sum_i k_i \leq 2$.

\medskip

Extracting a subsequence, not relabeled, we can assume that
$\gamma^n_1$ converges to an integer
rectifiable current $\gamma$. The intersection of the support
of $\gamma$ with 
$\partial \mathcal{B}_\sigma\setminus \{(0,0,\sigma),
(0,0, -\sigma)$
is then contained in the arcs $\alpha_i := P_{\theta_i} \cap
\partial \mathcal{B}_\sigma$. Therefore if we denote by $[[\alpha_i]]$ 
the current induced by $\alpha_i$ then we have
$$
\gamma \res \partial\mathcal{B}_\sigma \;=\; \sum_i h_i [[\alpha_i]]
$$
where the $h_i$ are integers.

On the other hand, $\gamma^n_1\res \mathcal{B}_\sigma$ is
given by the segment $\ell$. Therefore we conclude
that
$$
\gamma \res \mathcal{B}_\sigma \;=\; [[\ell]]\, .
$$
It turns out that
$$
\gamma \;=\; [[\ell]] + \sum_i h_i [[\alpha_i]]\, 
$$
and of course $\sum_i |h_i| \leq \sum_i k_i$.

Since $\partial \gamma =0$, we conclude that
$$
0\;=\; \partial [[l]] + \sum_i h_i \partial [[\alpha_i]]
\;=\; \delta_P- \delta_N + \sum_i h_i (\delta_N - \delta_P)
$$
where $N= (\sigma, 0, 0)$, $P=(-\sigma, 0, 0)$ and $\delta_X$ 
denotes the Dirac measure in the point $X$. Hence we conclude
$$
\left(1- \sum_i h_i\right) \delta_P - \left(1-\sum_i h_i\right) \delta_N 
\;=\; 0
$$
and therefore $\sum_i h_i = 1$. This implies that $\sum_i |h_i|$ is
odd. Since $\sum_i |h_i|\leq \sum_i k_i \leq 2$, we conclude
$\sum_i |h_i| = 1$. 

Therefore, $\gamma$ consists of the segment $\ell$ and an
arc, say, $\alpha_1$. Clearly,
$\gamma$ bounds $P_{\theta_1}$, which has area $\pi \sigma^2/2$.
Consider next the closed curve $\beta^n$ made by joining $\gamma^n_1
\cap \partial \mathcal{B}_\sigma$ and $-\alpha_1$.
These curves might have self--intersections, but they are close.
Moreover, they have bounded length and they converge, 
in the sense of currents, to the tivial current $\alpha_1-
\alpha_1 =0$.  

There are therefore domains 
$D^n\subset \mathcal{B}^+_\sigma$ such that
$\partial D^n = \beta^n$ and $\haus^2_e (D^n)\downarrow 0$.
It is not difficult to see that the union of the domains $D^n$
and of $P_{\theta_1}$ gives embedded disks $E^n$ bounding $\gamma^n_1$
and with area converging to $\pi \sigma^2/2$ (see Figure 
\ref{f:ultimopezzo}). Approximating these disks $E^n$ with
smooth embedded ones, we conclude that 
$$
\lim_{n\uparrow\infty} \mu_n \;\leq\; \frac{\pi}{2} \sigma^2\, .
$$
This shows that $\sum_i k_i\leq 1$. Hence the varifold
$W$ is either trivial or it consists of at most one half-disk.
Since it cannot be trivial by the considerations of
Subsections \ref{ss:tcones} and \ref{ss:tcones2}, we conclude
that $W$ consists in fact of exactly one half-disk.

\begin{figure}[htbp]
\begin{center}
    \input{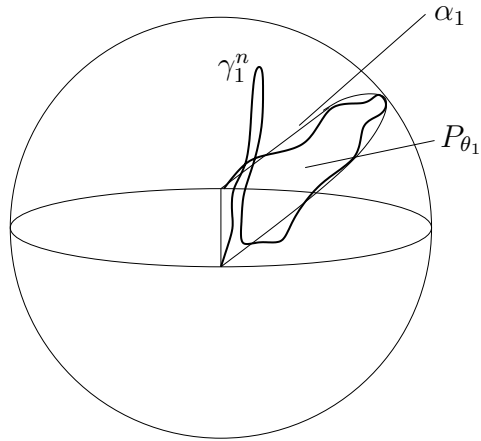}
    \caption{The curves $\gamma^n_1$ and $\alpha_1$.}
    \label{f:ultimopezzo}
\end{center}
\end{figure}

 \section{Proof of Proposition \ref{p:MSYbis}. Part V:
Convergence of connected components}\label{s:MSY5}

In this section we complete the proofs of Proposition \ref{p:MSYbis}
and Proposition \ref{p:MSYtris}. In particular, building on
Corollary \ref{c:multiplicity}, we show the following.

\begin{lemma}\label{l:components}
Let $\Sigma$ and $\Delta^k$ be as in Proposition \ref{p:MSYbis} (or
as in Proposition \ref{p:MSYtris}) and consider their varifold limit
$V$. According to Lemma \ref{l:interior}, Lemma \ref{l:boundary} and
Corollary \ref{c:multiplicity}, $V$ is a smooth stable minimal
surface with boundary $\partial \Delta = \partial \Sigma$ and with
multiplicity $1$. Let $\Gamma_1, \ldots, \Gamma_N$ be the connected
components of $\Delta$.

If $\tilde{\Delta}^k$ is an arbitrary union of connected components
of $\Delta^k$ which converges, in the sense of varifolds, to
a $W$, then $W$ is given by 
$\Gamma_{i_1}\cup \ldots \cup \Gamma_{i_l}$ for some $1\leq
i_1<i_2<\ldots < i_l\leq N$.
\end{lemma}

\begin{proof}
This lemma is indeed a simple
consequence of some known facts in geometric
measure theory. Fix 
a sequence $\tilde{\Delta}^k$ and a $W$ as in the
statement of the lemma. Note
that $\partial \tilde{\Delta}^k\subset \partial \Delta^k
= \partial \Sigma$. 

We can therefore apply 
the compactness of integer rectifiable currents and,
after a further extraction of subsequence, 
assume that the $\tilde{\Delta}^k$ are converging, as currents,
to an integer rectifiable current $T$ with boundary $\partial T$
which is the limit of the boundaries
$\partial \tilde{\Delta}^k$. Since
these boundaries are all contained in $\partial U$,
we conclude that $\partial T$ is also contained in $\partial U$. 
It is a known
fact in geometric measure theory that 
\begin{equation}\label{e:curren<varifold}
\|T\|\;\leq\; \|W\|\, .
\end{equation}
On
the other hand, 
\begin{equation}\label{e:monotvarifold}
\|W\|\;\leq\; \|V\|\;\leq\; \sum_i \haus^2 \res \Gamma_i\, .
\end{equation}
So $T$ is actually supported in the current given by the
union of the currents induced by the $\Gamma_i$'s, which
we denote by $[[\Gamma_i]]$. Since $\partial T$ and
$\partial \Gamma_i$ lie both on $\partial U$, a second
standard fact in geometric measure theory imply the existence
of integers $h_1, \ldots, h_N$ such that 
$$
T \;=\; \sum_{i=1}^N h_i [[\Gamma_i]]
$$
Therefore, 
\begin{equation}\label{e:massadiT}
\|T\| \;=\; \sum_i |h_i| \haus^2\res \Gamma_i\, .
\end{equation}
Hence, \eqref{e:curren<varifold},
\eqref{e:monotvarifold} and \eqref{e:massadiT}
give $h_i\in \{-1, 0, 1\}$. On the other hand,
since each $\partial \tilde{\Delta}^k$ is the union of
connected components of $\partial \Sigma$ (with positive orientation),
it turns out that $\partial T$ is the union of the currents
induced by some connected components of $\partial \Sigma$,
with {\em positive} orientation. Moreover, since $U$ is
a sufficiently small ball, by the maximum principle each 
surface $\Gamma_i$ must have nontrivial boundary. 
Hence, we conclude that $h_i\in \{0,1\}$. 

Arguing in the same way, we conclude that $\Delta^k\setminus
\tilde{\Delta}^k$ converge, as currents, to a current $T'$,
and, as varifolds, to a varifold $W'$ with the properties that
\begin{equation}
T' \;=\; \sum_{i=1}^N h'_i [[\Gamma_i]]
\end{equation}
\begin{equation}
\|T'\|\leq\|W'\|
\end{equation}
and $h'_i\in \{0,1\}$. Since $W+W' = V$, (and hence 
$\|W\|+\|W\|'=\|V\|$), we 
conclude that $h''_i = h'_i + h_i \in \{0,1\}$. On the other
hand, $\Delta^k$ converges, in the sense of currents, 
to $T+T'$, which is given by
\begin{equation}\label{e:somma}
T+T' \;=\; \sum_i (h_i+h'_i) [[\Gamma_i]]\, .
\end{equation}
Moreover, since $\partial \Delta^k = \partial \Sigma$,
\begin{equation}
[[\partial \Sigma]]\;=\; \partial (T+T') \;=\; \sum_i (h_i+h'_i)
[[\partial \Gamma_i]] \, .
\end{equation}
Since the $\partial \Gamma_i$ are all nonzero, disjoint
and contained in $\partial \Sigma$, we conclude that $h_i+h'_i=1$
for every $i$. 

Summarizing, we conclude that $\|V\| = \|W\|+\|W'\|\geq 
\|T\|+\|T'\|\geq \|T+T'\|= \|V\|$. This implies that $\|W\|+\|W'\|
= \|T\|+ \|T'\|$ and hence that $\|W\|=\|T\|$. Therefore
$$
\|W\| \;=\; \sum_i h_i \haus^2\res \Gamma_i
$$
and since $h_i\in \{0,1\}$, this last claim concludes the proof.
\end{proof}

\section{Considerations on \eqref{e:correct}
and \eqref{e:stronger}}
\label{s:discuss}

\subsection{Coverings}
In this subsection we discuss why \eqref{e:correct}
seems ultimately the correct estimate.
Fix a sequence $\{\Sigma^j_{t_j}\}$ which is $1/j$--a.m. in suffciently
small annuli and
assume for simplicity that each element is a smooth embedded surface
and that the varifold limit is given by
$$
\Gamma \;=\; \sum_{\Gamma^i\in \mathcal{O}}
 n_i \Gamma^i + \sum_{\Gamma^i\in \mathcal{N}} n_i \Gamma^i\,.
$$
Then, one expects that, after appropriate surgeries (which can only
bring the genus down) $\Sigma^j_{t_j}$ split into three groups.
\begin{itemize}
\item The first group consists of
$$
m_1\;=\;\sum_{\Gamma^i\in \mathcal{O}} n_i
$$
surfaces, each isotopic to a $\Gamma^i\in \mathcal{O}$;
\item The second group consists of
$$
m_2\;=\; \frac{1}{2}\sum_{\Gamma^i\in \mathcal{N}} n_i\,
$$
surfaces, each isotopic to the boundary of a regular tubular
neighborhood of $\Gamma^i\in \mathcal{N}$, (which is a double cover
of $\Gamma^i$);
\item The sum of the areas of the the third group
vahishes as $j\uparrow \infty$.
\end{itemize}
As a consequence one would conclude that $n_i$ is even whenever
$\Gamma^i\in \mathcal{N}$ and that \eqref{e:correct} holds.

\medskip

The type of convergence described above is exactly the one proved by
Meeks, Simon and Yau in \cite{MSY} for sequences of surfaces which
are minimizing in a given isotopy class. The key ingredients of
their proof is the $\gamma$--reduction and the techniques set forth
by Almgren and Simon in \cite{AS} to discuss sequences of minimizing
disks. However, in their situation there is a fundamental advantage:
when the sequence $\{\Sigma^j\}$ is minimizing in a given isotopy
class, one can perform the $\gamma$--reduction ``globally'', and
conclude that, after a finite number of surgeries which do not
increase the genus, there is a constant $\sigma>0$ with the
following property:
\begin{itemize}
\item For any ball $B$ with radius $\sigma$, each
curve in $\partial B\cap \Sigma^j$ bounds a small disk in
$\Sigma^j$.
\end{itemize}

In the case of min--max sequences, their weak $1/j$--almost
minimizing property on subsets of the ambient manifold allows to
perform the $\gamma$--reduction only to surfaces which are
appropriate local modifications of the $\Sigma^j$'s, see the
Squeezing Lemma of Section \ref{s:MSY2} and the modified 
$\gamma$--reduction of Section \ref{s:MSY2}. Unfortunately,
the size of the open sets where this can be done depends on $j$.
In order to show that the picture above holds, it seems necessary
to work directly in open sets of a fixed size.

\subsection{An example} In this section we show that
\eqref{e:stronger} cannot hold for sequences
which are $1/j$--a.m.. Consider in particular the 
manifold $M = ]-1,1[\times \SS^2$ with the
standard product metric. 
We parameterize $\SS^2$ with 
$\{|\omega|=1 : \omega\in {\bf R}^3\}$.  
Consider on $M$ the orientation--preserving
diffeomorphism $\varphi: (t, \omega) \mapsto (-t, -\omega)$
and the equivalence relation $x\sim y$ if $x=y$ or 
$x= \varphi (y)$. Let $N= M/_\sim$ be the quotient manifold,
which is an oriented Riemannian manifold, and consider 
the projection $\pi: M\to N$, which is a local isometry.
Clearly, $\Gamma:= \pi (\{1\}\times \SS^2)$ is an embedded  
$2$--dimensional (real) projective plane. Consider a sequence
$t_j\downarrow 1$. Then, each $\Lambda^j:=\{t_j\}\times \SS^2$ is a
totally geodesic surface in $M$ and, therefore, $\Sigma^j=
\pi (\Lambda_j)$ is totally
geodesic as well. Let $r$ be the injectivity radius
of $N$ and consider a smooth open set $U\subset N$ 
with diameter smaller than $r$ such that
$\partial U$ intersects $\Sigma^j$ transversally.
Then $\Sigma^j\cap U$ is the unique
area--minimizing surface spanning $\partial U\cap \Sigma^j$.

Hence, the sequence of surfaces $\{\Sigma^j\}$
is $1/j$--a.m. in sufficiently
small annuli of $N$. Each $\Sigma^j$ is a smooth
embedded minimal sphere and $\Sigma^j$ converges,
in the sense of varifolds, to $2 \Gamma$. 
Since $\gen (\Sigma^j) = 0$
and $\gen (\Gamma)=1$, the inequality
$$
\gen (\Gamma) \;\leq\; \liminf_{j\uparrow \infty}
\gen (\Sigma^j)\, ,
$$
which corresponds to \eqref{e:stronger}, fails in this case.

\appendix

\section{Proof of Lemma \ref{l:monot}}\label{a:monot}

\begin{proof} Let $\Sigma$ be a smooth minimal surface
with $\partial \Sigma\subset \partial B_\sigma (x)$, where
$\sigma<r_0$ and $r_0$ is a positive constant to
be chosen later. We recall that, for every vector field 
$X\in C^1_c (B_\sigma (x))$, we have
\begin{equation}\label{e:staz}
\int_{B_\sigma (x)} {\rm div}_\Sigma X \;=\; 0\, .
\end{equation}
We assume $r_0< {\rm Inj}\, (M)$ (the injectivity radius
of $M$) and we use geodesic coordinates centered at $x$.
For every $y\in B_\sigma (x)$ we denote by $r (y)$
the geodesic distance between $y$ and $x$. Recall that
$r$ is Lipschitz on $B_\sigma (x)$ and $C^\infty$ in 
$B_\sigma (x)\setminus \{x\}$,
and that $|\nabla r|=1$, where $|\nabla r|= 
\sqrt{g (\nabla r, \nabla r)}$.

We let $\gamma\in C^1 ([0,1])$ be a cut-off function,
i.e. $\gamma=0$ in a neighborhood of $1$ and $\gamma=1$ in
a neighborhood of $0$. We set
$X = \gamma (r) r \nabla r = \gamma (r) \nabla \frac{|r|^2}{2}$.
Thus, $X\in C^\infty_c (B_\sigma (x))$
and from \eqref{e:staz} we compute
\begin{equation}\label{e:staz2}
0\;=\; \int_\Sigma \gamma (r)\, {\rm div}_\Sigma\, (r \nabla r)
+ \int_\Sigma r\, \gamma' (r) \sum_i \partial_{e_i} r\, g (\nabla r, e_i)\, ,
\end{equation}
where $\{e_1, e_2\}$ is an orthonormal frame on $T \Sigma$.
Clearly 
\begin{equation}\label{e:conti1}
\sum_i \partial_{e_i} r\, g (\nabla r, e_i)=
\sum_i (\partial_{e_i} r)^2 = |\nabla_\Sigma r|^2 = |\nabla r|^2
- |\nabla^\perp r|^2 = 1 - |\nabla^\perp r|^2\, ,
\end{equation}
where $\nabla^\perp r$ denotes
the projection of $\nabla r$ on the normal bundle to $\Sigma$. 
Moreover, let $\nabla^e$ be the euclidean connection
in the geodesic coordinates and consider a $2$-d plane $\pi$ in $T_y M$,
for $y\in B_\sigma (x)$. Then
$$
{\rm div}_\pi\, (r (y)\, \nabla r (y))
- {\rm div}^e_\pi\, (|y|\, \nabla^e |y|) = O (|y|) = O (\sigma)\, .
$$
Since ${\rm div}^e_\pi\, (|y|\, \nabla^e |y|)=2$, 
we conclude the existence of a constant $C$ such that
\begin{equation}\label{e:errore}
\left|\int_\Sigma \gamma (r) {\rm div}_\Sigma (r \nabla r)
- 2 \int_\Sigma \gamma (r)\right|\;\leq\; C \|\gamma\|_\infty \sigma \haus^2
(\Sigma\cap B_\sigma (x))\, .
\end{equation}
Inserting \eqref{e:conti1} and \eqref{e:errore} in
\eqref{e:staz2}, we conclude
\begin{equation}\label{e:simon}
\int_\Sigma 2\, \gamma (r) + \int_\Sigma r\, \gamma' (r)
\;=\; \int_\Sigma r\, \gamma' (r)\, |\nabla^\perp r|^2 + {\rm Err}
\end{equation}
where, if we test with functions $\gamma$ taking values in $[0,1]$,
we have
\begin{equation}\label{e:errore2}
|{\rm Err}| \leq C \sigma \haus^2 (\Sigma\cap B_\sigma (x))\, .
\end{equation}
We test now \eqref{e:simon} with functions
taking values in $[0,1]$ and approximating the characteristic functions
of the interval $[0,\sigma]$. Following the computations of pages 83-84
of \cite{Si}, we conclude
\begin{equation}
\left.\frac{d}{d\rho} \left( \rho^{-2} \haus^2 (\Sigma\cap B_\rho (x))\right)
\right|_{\rho =\sigma} \;=\; \left.\frac{d}{d\rho} \left(\int_{\Sigma\cap B_\rho (x)}
\frac{|\nabla^\perp r|^2}{r^2}\right)\right|_{\rho=\sigma}
+ \sigma^{-3} {\rm Err}\, .
\end{equation}
Straightforward computations lead to
\begin{equation}\label{e:conti2}
\haus^2 (\Sigma\cap B_\sigma (x))
\;=\; \underbrace{\frac{\sigma}{2}\left.\frac{d}{d\rho} \left(\haus^2 (\Sigma\cap B_\rho (x))\right)
\right|_{\rho =\sigma} - \frac{\sigma^3}{2} 
\left.\frac{d}{d\rho} \left(\int_{\Sigma\cap B_\rho (x)}
\frac{|\nabla^\perp r|^2}{r^2}\right)\right|_{\rho=\sigma}}_{=(A)}
+ {\rm Err}\, .
\end{equation}
Moreover, by the coarea formula, we have
\begin{eqnarray}
(A) &=& \frac{\sigma}{2} \int_{\partial B_\sigma 
(x)\cap \Sigma} \frac{1}{|\nabla_\Sigma r|}
- \frac{\sigma^3}{2} \int_{\partial B_\sigma (x) \cap \Sigma}
\frac{|\nabla^\perp r|^2}{\sigma^2 |\nabla_\Sigma r|}
\;=\;\frac{\sigma}{2}\int_{\partial \Sigma} \frac{1-|\nabla^\perp r|^2}{|\nabla_\Sigma r|}
\nonumber\\
&=&\frac{\sigma}{2}\int_{\partial \Sigma} |\nabla_\Sigma r|
\;\leq\; \frac{\sigma}{2}\Length\, 
(\partial \Sigma)\label{e:(A)}\, .
\end{eqnarray}
Inserting \eqref{e:(A)} into \eqref{e:conti2}, we conclude
that
\begin{equation}
\haus^2 (\Sigma\cap B_\sigma (x))
\;\leq\; \frac{\sigma}{2}\Length\, 
(\partial \Sigma) + |{\rm Err}|\, ,
\end{equation}
which, taking into account \eqref{e:errore2}, becomes
\begin{equation}
(1- C \sigma)\haus^2 (\Sigma\cap B_\sigma (x))
\;\leq\; \frac{\sigma}{2}\Length\, 
(\partial \Sigma)\, .
\end{equation}
So, for $r_0 < \min\{\Inj (M), (2C)^{-1}\}$ we get
\eqref{e:monot}.
\end{proof}

\section{Proof of Lemma \ref{l:technical_projection}}\label{a:proj}

\begin{proof} Let $d_e (y)$ be the euclidean distance of $y$ to
$\overline{U}$ and $d (y)$ the geodesic distance of $y$ to
$\overline{f(U)}$. The function $d_e$ is $C^2$ and uniformly 
convex on the closure of $\mathcal{B}_1\setminus U$. Therefore,
if $\eps_0$ is sufficiently small, the function $d$ is uniformly
convex on the closure of $B_\eps (x)\setminus f(U)$. 
Let now $y_0\in B_\eps (x)\setminus \overline{f (U)}$. 
In order to find $\pi (x)$ it suffices to follow the flow line
of the ODE $\dot{y} = - \nabla d (y)/|\nabla d (y)|^2$, with
initial condition $y (0)=y_0$, until the line hits $\overline{f(U)}$.
Thus, the inequality $|\nabla \pi (x)|<1$ follows from
Lemma 1 of \cite{bangert}. On the other hand, $\pi (x) =x$
on $\overline{f(U)}$, and therefore the map is Lipschitz with constant
$1$. 
\end{proof}

\section{A simple topological fact}\label{a:isotopie}

We summarize the topological fact used in (a) of
Section \ref{ss:disks} in the following lemma.

\begin{lemma}\label{l:technical_topological}
Condider a smooth $2$--dimensional surface $\Sigma \subset
\mathcal{B}_1$ with smooth boundary $\partial \Sigma \subset
\partial \mathcal{B}_1$. Let $\Gamma\subset \mathcal{B}_1$ 
is a smooth surface with $\partial \Gamma = \partial \Sigma$
consisting of disjoint embedded disks. Then there exists
a smooth map $\Phi : [0,1[\times \overline{\mathcal{B}}_1 
\to \overline{\mathcal{B}}_1$
such that
\begin{itemize}
\item[(i)] $\Phi (0, \cdot)$ is the identity and $\Phi (t, \cdot)$
is a diffeomorphism for every $t$;
\item[(ii)] For every $t$ there exists a neighborhood $U_t$ of
$\partial \mathcal{B}_1$ such that $\Phi (t,x)=x$ for every
$x\in U_t$;
\item[(iii)] $\Phi (t, \Sigma)$ converges to $\Gamma$ in the 
sense of varifolds as $t\to 1$.
\end{itemize}
\end{lemma}

\begin{proof} The proof consists of two steps. In the first
one we show the existence of a surface $\Gamma'$ and
of a map $\Psi: [0,1[\times \overline{\mathcal{B}}_1\to 
\overline{\mathcal{B}}_1$ such that
\begin{itemize}
\item $\partial \Gamma' = \partial \Sigma$,
\item $\Gamma'$ consists of disjoint embedded disks,
\item $\Psi$ satisfies (i) and (ii),
\item $\Psi (t, \Sigma)\to \Gamma'$ as $t\to 1$. 
\end{itemize}
In the second we show the existence of a $\tilde{\Psi}:
[0,1[\times \overline{\mathcal{B}}_1 \to \overline{\mathcal{B}}_1$
such that (i) and (ii) hold and
$\tilde{\Psi} (t, \Gamma')\to \Gamma$ as $t\to 1$.

In order to complete the proof from these two steps,
consider the map $\tilde{\Phi} (s,t, x) = \tilde{\Psi} (t, 
\Psi (s,x))$. Then, for every smooth $g:[0,1[\to [0,1[$ with
$g (0)=0$, the map $\Phi (t, x) = \tilde{\Phi} (g(t), t, x)$
satisfies (i) and (ii) of the Lemma. Next, for any fixed $t$,
if $s$ is sufficiently close to $1$, then
$\tilde{\Phi} (s,t,\Sigma)$ is close, in the sense of varifolds,
to $\tilde{\Psi} (t, \Gamma')$. This allows to find 
a piecewise constant function $h:[0,1[\to [0,1[$ such that
$$
\lim_{t\to 1} \tilde{\Phi} (g(t), t, \Sigma) \;=\; \Gamma
\qquad \qquad \mbox{(in the sense of varifolds)}
$$ 
whenever $g\geq h$ in a neighborhood of $1$. If we choose,
therefore, a smooth $g: [0,1[\to [0,1[$ with $g(0)=0$ and $g\geq h$
on $[1/2, 1[$, the map $\Phi (t,x)=\tilde{\Phi} (g(t),t,x)$ satisfies
all the requirements of the lemma.

We now come to the existence of the maps $\Psi$ and $\tilde{\Psi}$.

\medskip

{\bf Existence of $\Psi$.} Let $\mathcal{G}$ be the set of
all surfaces $\Gamma'$
which can be obtained as $\lim_{t\to 1} \Psi (t, \Sigma)$
for maps $\Psi$ satisfying (i) and (ii). It is easy to
see that any $\Gamma'$ which is obtained from $\Sigma$ through
surgery as in Definition \ref{d:surgery} is contained in $\mathcal{G}$.
Let $\gen_0$ be the smallest genus of a surface contained
in $\mathcal{G}$. It is then a standard fact that 
$\gen (\Gamma') = \gen_0$ if and only if the surface is incompressible.
However, if this holds, then the first homotopy group
of $\Gamma'$ is mapped injectively in the homotopy group of
$\mathcal{B}_1$ (see for instance \cite{Jaco}). Therefore
there is a $\Gamma'\in \mathcal{G}$ which consists of disjoint 
embedded disks and spheres. The embedded spheres can be 
further removed,
yielding a $\Gamma'\in \mathcal{G}$ consisting only of disjoint 
embedded disks. 

\medskip

{\bf Existence of $\tilde{\Psi}$.} Note that each connected component
of $\mathcal{B}_1\setminus \Gamma'$ (and of $\mathcal{B}_1\setminus
\Gamma$) is a, piecewise smooth, embedded sphere.
Therefore the claim can be easily proved by induction
from the case in which $\Gamma$ and $\Gamma'$ consist both of
a single embedded disk. This is, however, a standard fact
(see once again \cite{Jaco}). 
\end{proof}

{\bf Acknoweldgments} Both authors have been supported by a grant of
the Swiss National Foundation.

\nocite{*}
\bibliographystyle{plain}
\bibliography{bibliografia}

\end{document}